\documentclass[10pt,a4paper]{amsart}
\usepackage{amssymb,amsmath}
\usepackage{graphicx}
\usepackage{bm,bbm}
\usepackage{color}
\usepackage{tikz}
\usepackage{ifthen}
\usetikzlibrary{snakes}
\usetikzlibrary{patterns}
\usepackage{pgfplots}
\usepackage{enumitem}
\usepackage{array,blkarray}
\usepackage{bigdelim}
\usepackage{listings} 
\usepackage[hidelinks]{hyperref}

\newcommand{\simZ}{\sim} 
\usepackage{mathtools}  
\newtheorem{theorem}{Theorem}

\newtheorem{corollary}[theorem]{Corollary}
\theoremstyle{definition}

\newtheorem{example}[theorem]{Example}
\theoremstyle{lemma}
\newtheorem{lemma}[theorem]{Lemma}

\theoremstyle{remark}
\newtheorem{remark}[theorem]{Remark}

\newtheorem{assumption}[theorem]{Assumption}
\numberwithin{theorem}{section}
\numberwithin{equation}{section}
\numberwithin{table}{section}
\numberwithin{figure}{section}
\newcommand{\quotes}[1]{``#1''}
\definecolor{myBlue}{RGB}{113,104,238} 
\definecolor{myGreen}{RGB}{154,205,50} 
\definecolor{myGreen2}{RGB}{114,175,30} 
\definecolor{myRed}{RGB}{180,50,50}  
\definecolor{myOrange}{RGB}{225,92,22} 
\definecolor{lgray}{RGB}{200,200,200} 
\definecolor{llgray}{RGB}{155,155,155} 
\newcommand\N{\mathbb N}

\newcommand\R{\mathbb R}

\DeclareMathOperator{\sspan}{span}
\DeclareMathOperator{\supp}{supp}
\DeclareMathOperator{\id}{id}

\DeclareMathOperator{\range}{im}

\DeclareMathOperator{\tol}{tol}

\DeclareMathOperator{\gap}{gap}
\DeclareMathOperator{\plog}{polylog}

\newcommand\calA{\mathcal A}
\newcommand\calB{\mathcal B}
\newcommand\calC{\mathcal C}
\newcommand\calH{\mathcal H}
\newcommand\calI{\mathcal I}

\newcommand\calN{\mathcal N}
\newcommand\calO{\mathcal O}
\newcommand\calP{\mathcal P}
\newcommand\calQ{\mathcal Q}
\newcommand\calT{\mathcal T}
\newcommand\calV{\mathcal V}
\newcommand\calW{\mathcal W}
\newcommand\calQa{\mathcal Q_\alpha}
\newcommand\calQb{\mathcal Q_\beta}
\newcommand\calQnu{\mathcal Q_\nu}
\newcommand\calPP{\mathcal R}		

\newcommand{\hook}{\ensuremath{\hookrightarrow}}
\newcommand\eps{\varepsilon}

\def\ds{\,\text{d}s}

\def\dx{\,\text{d}x}

\def\dy{\,\text{d}y}
\def\dz{\,\text{d}z}

\newcommand\E{E}  

\newcommand\V{\calV}   
\newcommand\W{\calW}   
\newcommand\Hper{H^1_\text{per}(D)}
\newcommand\Vvert{{|\hskip-0.9pt|\hskip-0.9pt|}}
\newcommand\bigVvert{\big|\hskip-0.9pt\big|\hskip-0.9pt\big|}

\newcommand\Oa{{D_\alpha}}  
\newcommand\Ob{{D_\beta}}  
\newcommand\Onu{{D_\nu}}  
  
\newcommand\calTa{{\calT_\alpha}}
\newcommand\calTb{{\calT_\beta}}

\newcommand\calNb{{\calN_\text{rem}}}

\newcommand\lamz{{\lambda_z}}

\newcommand\PQ{\calP_Q}

\newcommand\Pz{\calP_z}
\newcommand\tPz{{\tilde\calP}_z}
\newcommand{\ISZ}{I^\text{sz}}

\newcommand\cl{c_L}
\DeclareMathOperator{\err}{err}
\definecolor{myBlue3}{RGB}{60,124,155} 
\definecolor{lGray}{RGB}{180,180,180}

\begin{document}
\title[Quantitative Anderson localization of Schr\"odinger states]{Quantitative Anderson localization of Schr\"odinger eigenstates under disorder potentials}
\author[]{R.~Altmann$^*$, P.~Henning$^{\dagger}$, D.~Peterseim$^*$}
\address{${}^{*}$ Department of Mathematics, University of Augsburg, Universit\"atsstr.~14, 86159 Augsburg, Germany}
\address{${}^{\dagger}$ Department of Mathematics, KTH Royal Institute of Technology, SE-100 44 Stockholm, Sweden}
\email{\{robert.altmann, daniel.peterseim\}@math.uni-augsburg.de, pathe@kth.se}
\thanks{R.~Altmann acknowledges support from the Einstein Center ECMath via project OT10. P.~Henning acknowledges funding by the Swedish Research Council (grant 2016-03339). D.~Peterseim acknowledges support by Deutsche Forschungsgemeinschaft in the Priority Program 1748 {\it Reliable simulation techniques in solid mechanics} (PE2143/2-2). Further, the authors thank the Hausdorff Institute for Mathematics in Bonn for the kind hospitality during the trimester program on multiscale problems in 2017}
\date{\today}
%
%
\begin{abstract}
This paper analyzes spectral properties of linear Schr\"odinger operators under oscillatory high-amplitude potentials on bounded domains. Depending on the degree of disorder, we prove the existence of spectral gaps among the lowermost eigenvalues and the emergence of exponentially localized states. We quantify the rate of decay in terms of geometric parameters that characterize the potential. The proofs are based on the convergence theory of iterative solvers for eigenvalue problems and their optimal local operator preconditioning by domain decomposition.
\end{abstract}
%
%
\maketitle
\setcounter{tocdepth}{2}
%
%
%
\section{Introduction}
Linear Schr\"odinger operators $\calH$ of the form $\calH:=-\triangle\,+\,V$ composed of the $d$-dimen\-sio\-nal Laplacian $\triangle$ and a non-negative potential $V$ are an important building block for the mathematical modeling of quantum-physical processes related to ultracold bosonic or photonic gases -- so-called Bose-Einstein condensates \cite{Gro61,Pit61,LSY01,Alaeian_2017}. The cases where the potential $V$ exhibits disorder \cite{NBP13} or represents quantum arrays in the context of Josephson oscillations \cite{WWW98,ZSL98} have raised particular attention. Surprising phenomena such as Anderson localization of eigenfunctions \cite{And58} are connected to such oscillatory and disordered potentials. Anderson localization refers to exponentially localized low-energy stationary states which are caused exactly by the interplay of disorder (randomness) and high amplitude (contrast) of the potential trap. For matter waves, this has been experimentally observed, cf.~\cite{2008Natur.453..891B,2008Natur.453..895R}. While localization is understood qualitatively in an asymptotic sense or can be justified a posteriori in the mathematical model, the a priori prediction of the localization effect in terms of geometric parameters that characterize the potential and its degree of disorder remained open and this paper aims to close this gap. 
\smallskip

We consider a bounded domain $D\subset\R^d$ and prototypical disorder potentials that vary randomly between two values $\alpha$ and $\beta$, where $\alpha \ll \tfrac{1}{\eps^2} \le \beta$ on a small scale $\eps$. The main result of the paper states that sufficiently disordered potentials imply the existence of $K$ points $z_1,\ldots,z_K\in D$ and some rate $c>0$, which only depends logarithmically on~$\eps$, such that the ground state $u_1$ satisfies 
\begin{equation*}
  \Vvert u_1\Vvert_{D \setminus \bigcup_{j=1}^K B(z_j;\,\eps k^2)}
  \lesssim e^{-ck}\, \Vvert u_1\Vvert_D
\end{equation*}
for all $k=1,2,\ldots$. 
Here $\Vvert\cdot\Vvert_\omega^2:=\Vert \nabla v\Vert_\omega^2 + \Vert V^{1/2}\cdot\Vert_\omega^2$ denotes the energy norm and $B(z;\, r)$ the ball of radius~$r$ centered at~$z$ in the sup norm. This is illustrated in Figure~\ref{fig:locstates} for a prototypical disorder potential. Similar results hold true for eigenstates in a certain range of energies at the bottom of the spectrum. 

The proof of existence of exponentially localized states consists of three main steps. The first step is the quantification of the exponential decay of the Green's function associated with $\calH$ in terms of the oscillation length and the amplitude of the potential in Section~\ref{sec:precond}. Disorder is not essential at this point. This novel result is inspired by numerical homogenization for arbitrarily rough diffusion tensors \cite{MalP14,HenP13} which in turn is closely connected to the exponential decay of the corrector Green's function \cite{Pet16}. An elegant new proof of the latter decay property was later given in \cite{KorY16, KorPY18}. This pioneering approach employs classical results from domain decomposition and the convergence theory of iterative solvers for linear systems arising from the discretization of partial differential equations.

\begin{figure}
	\centering
	\includegraphics[width=0.29\linewidth]{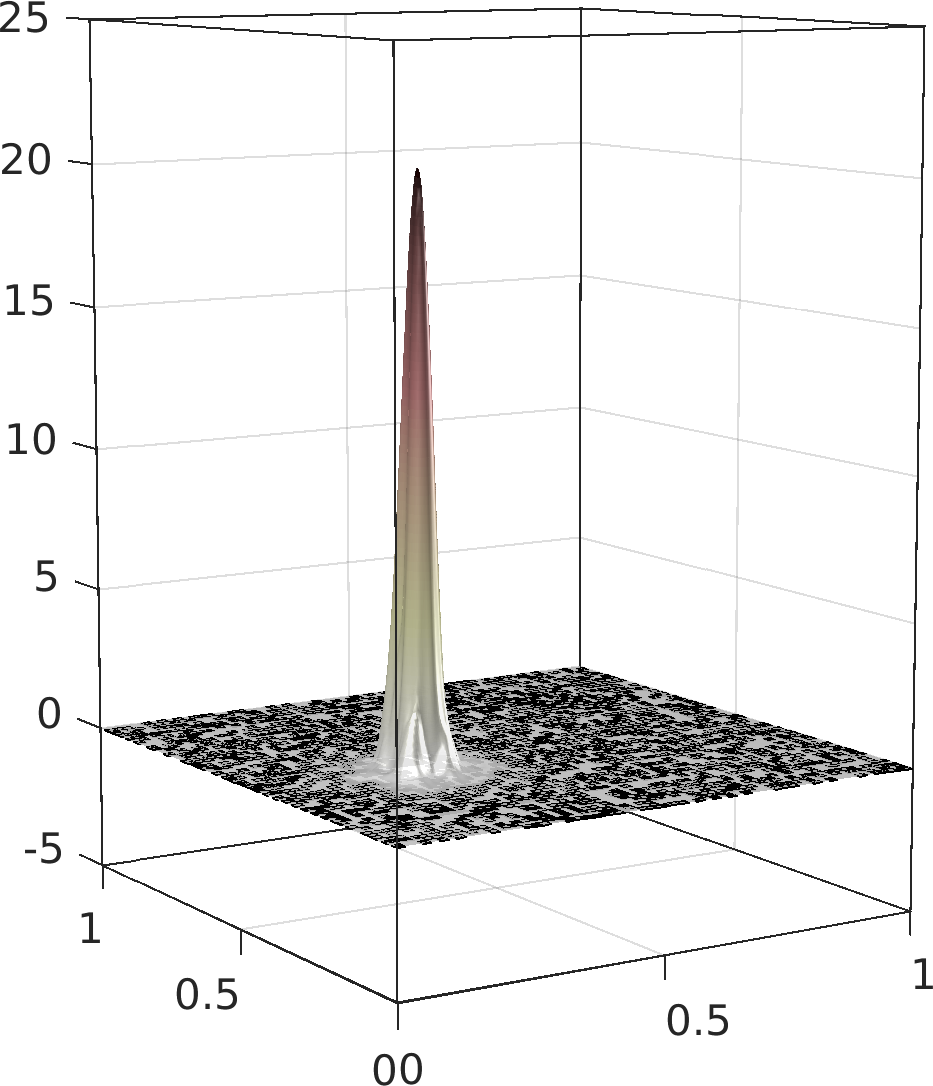}\hspace{2ex}
	\includegraphics[width=0.32\linewidth]{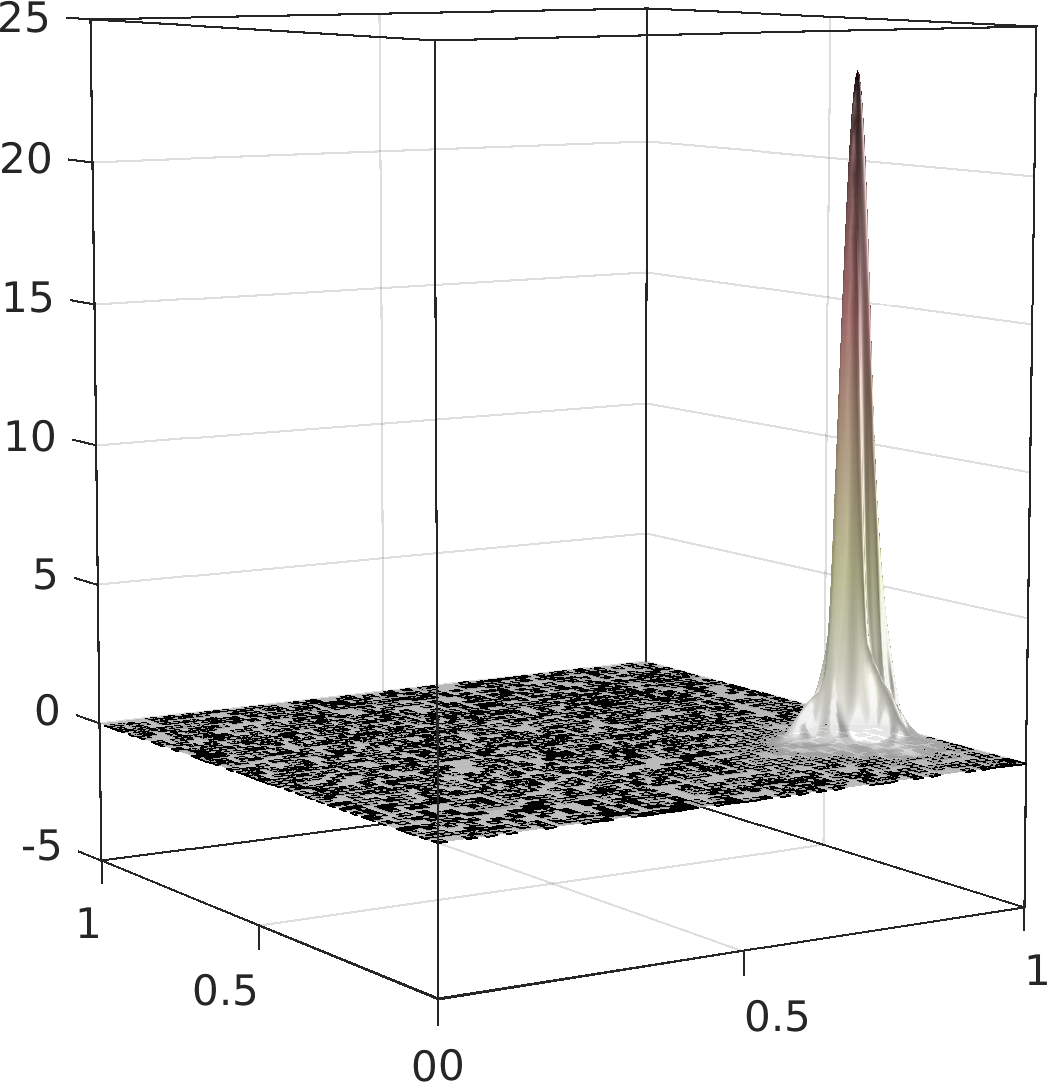}\hspace{2ex}
	\includegraphics[width=0.32\linewidth]{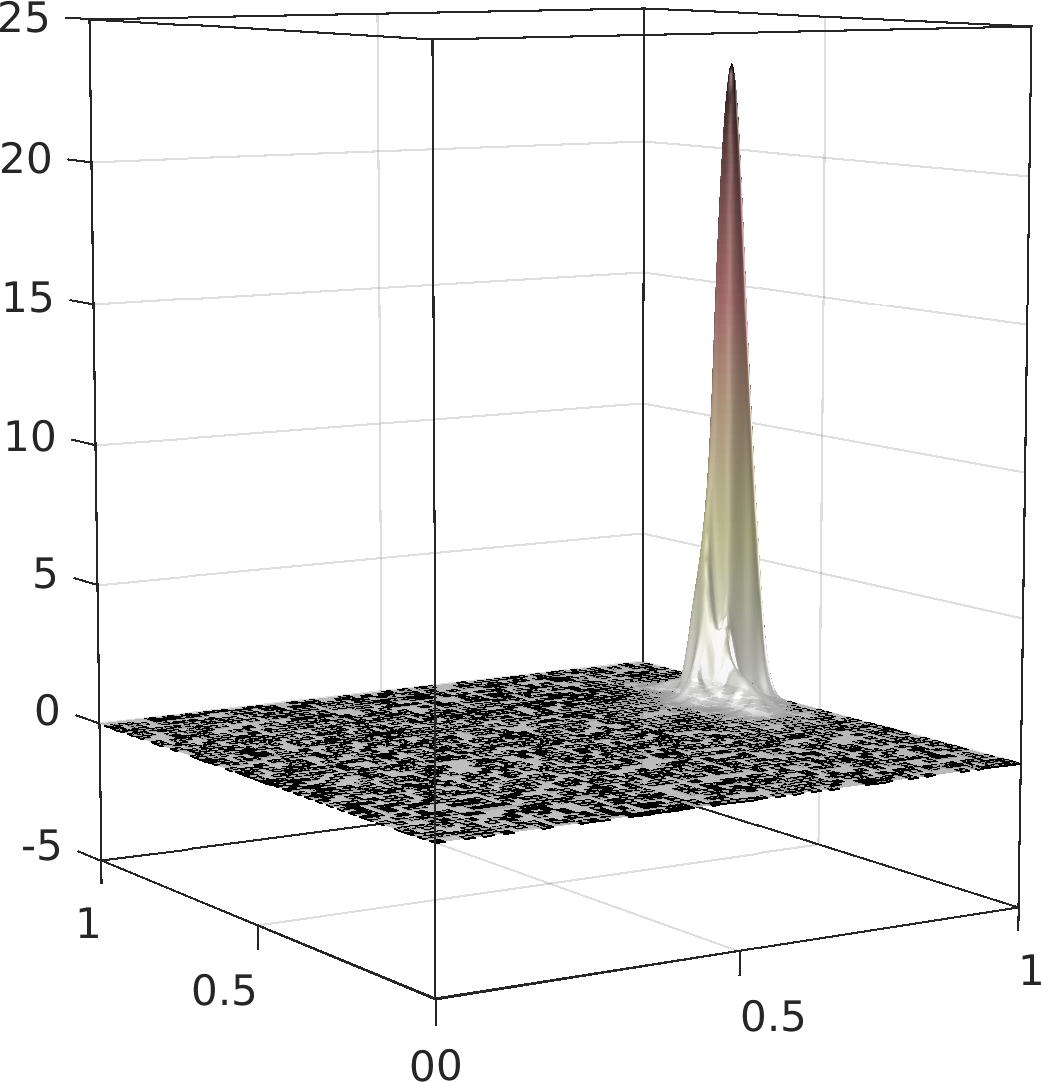}\\[1ex]
	\includegraphics[width=0.32\linewidth]{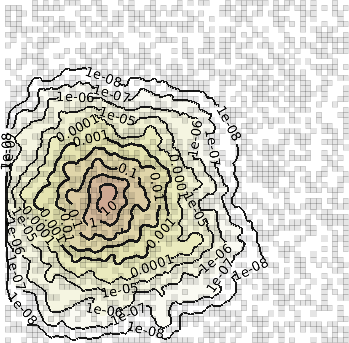}\hspace{1ex}
	\includegraphics[width=0.32\linewidth]{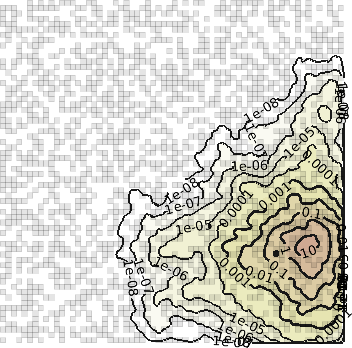}\hspace{1ex}
	\includegraphics[width=0.32\linewidth]{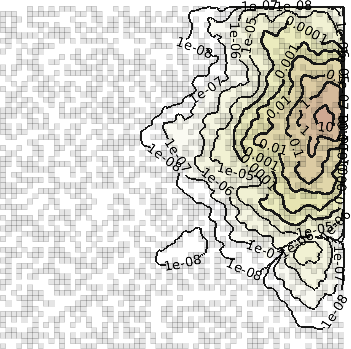}
	\caption{Schr\"odinger eigenstates (homogenous Dirichlet Boundary condition) associated to the three smallest eigenvalues (from left to right) in a disorder potential. Top row: Graphs of eigenstates. Bottom row: Isolines of moduli representing exponential decay in scales of $\varepsilon$. The disorder potential is a random checkerboard on a Cartesian mesh of the unit square of width $\epsilon=2^{-6}$ taking values $\beta = 4/\varepsilon^2$ (black) and $\alpha = 1$ (white).
	\label{fig:locstates}}
\end{figure}

\begin{figure}
	\centering
	\includegraphics[width=0.4\linewidth]{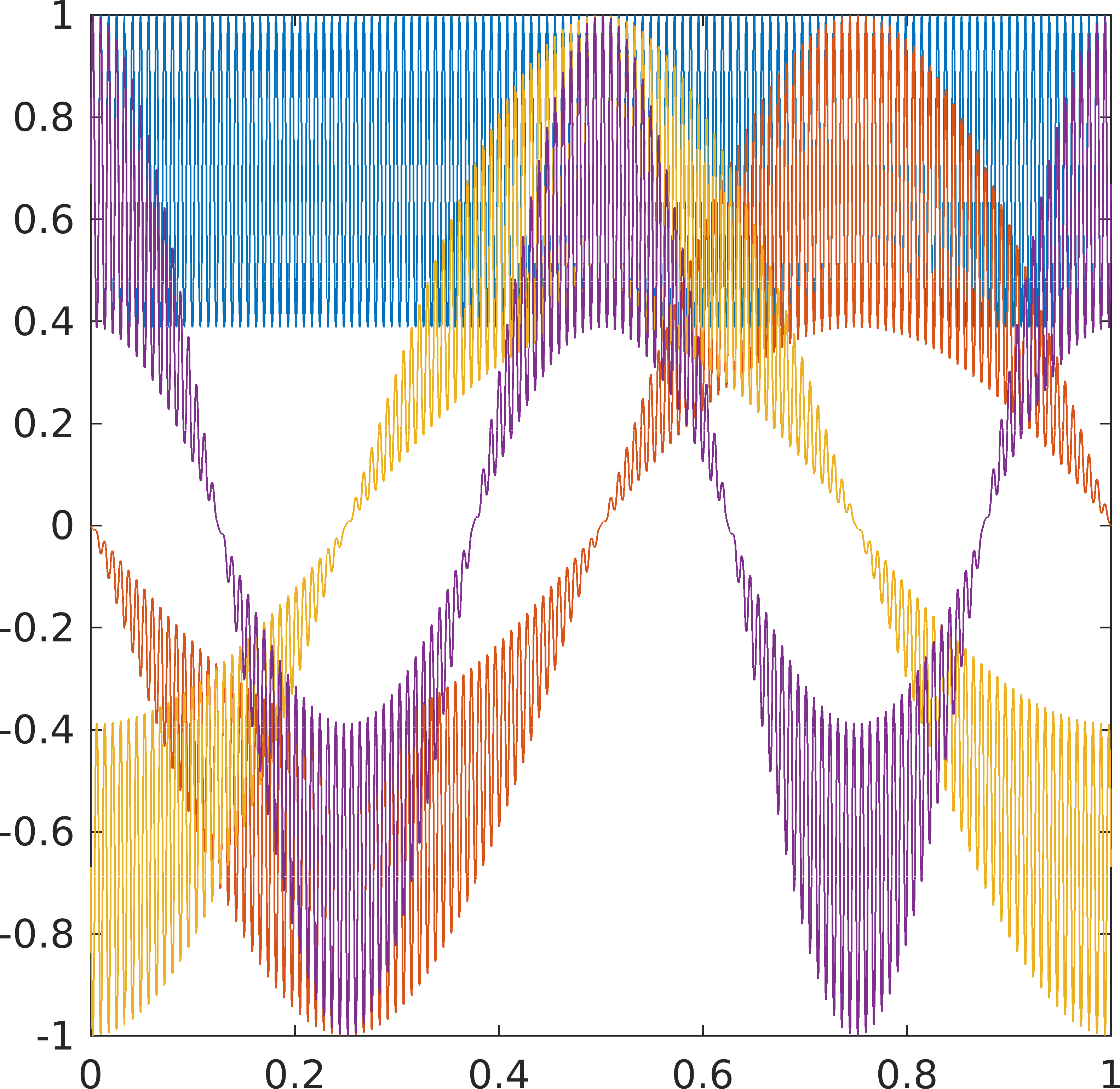}\hspace{2ex}
	\includegraphics[width=0.4\linewidth]{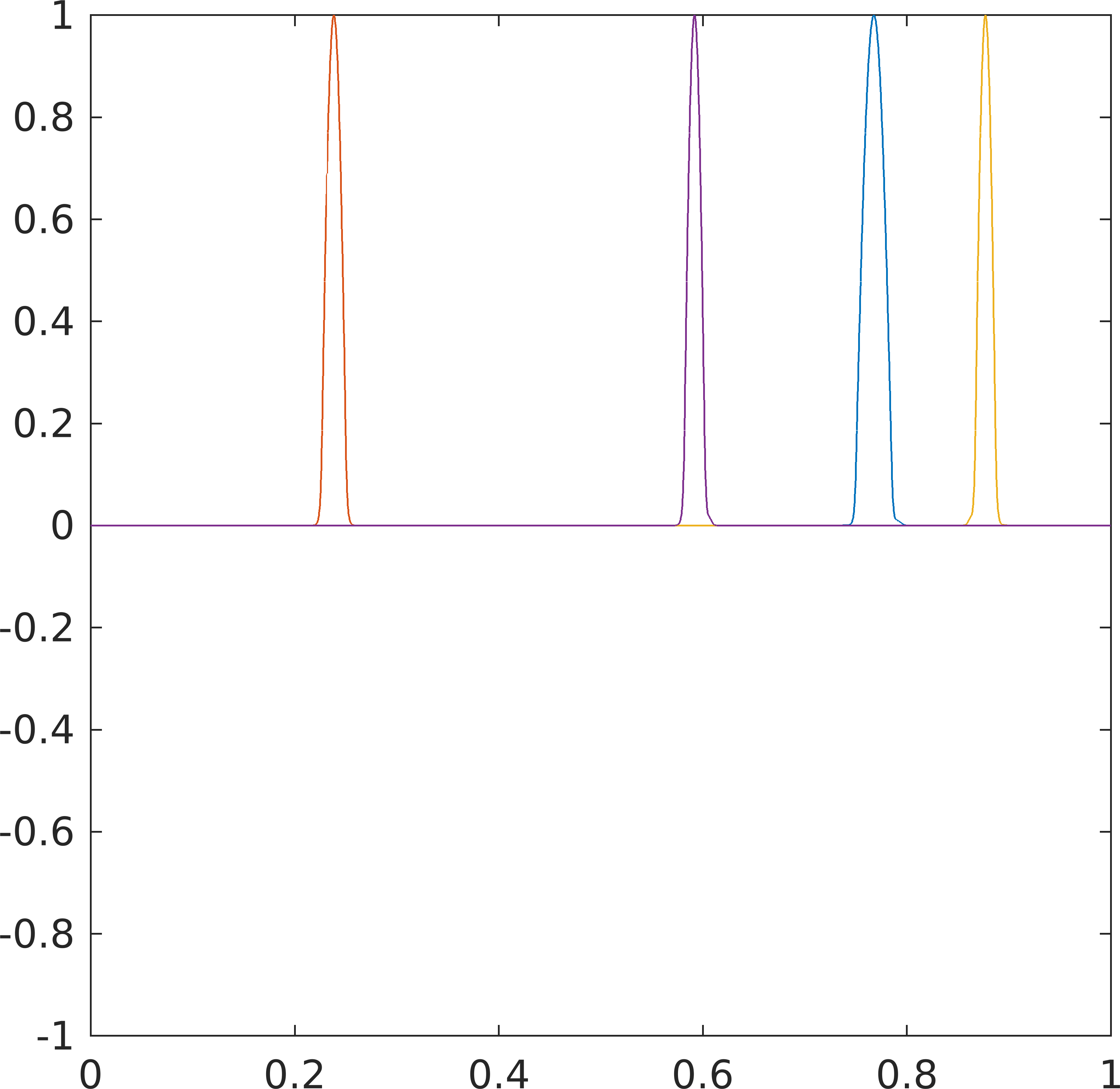}\\[1ex]
	\includegraphics[width=0.4\linewidth]{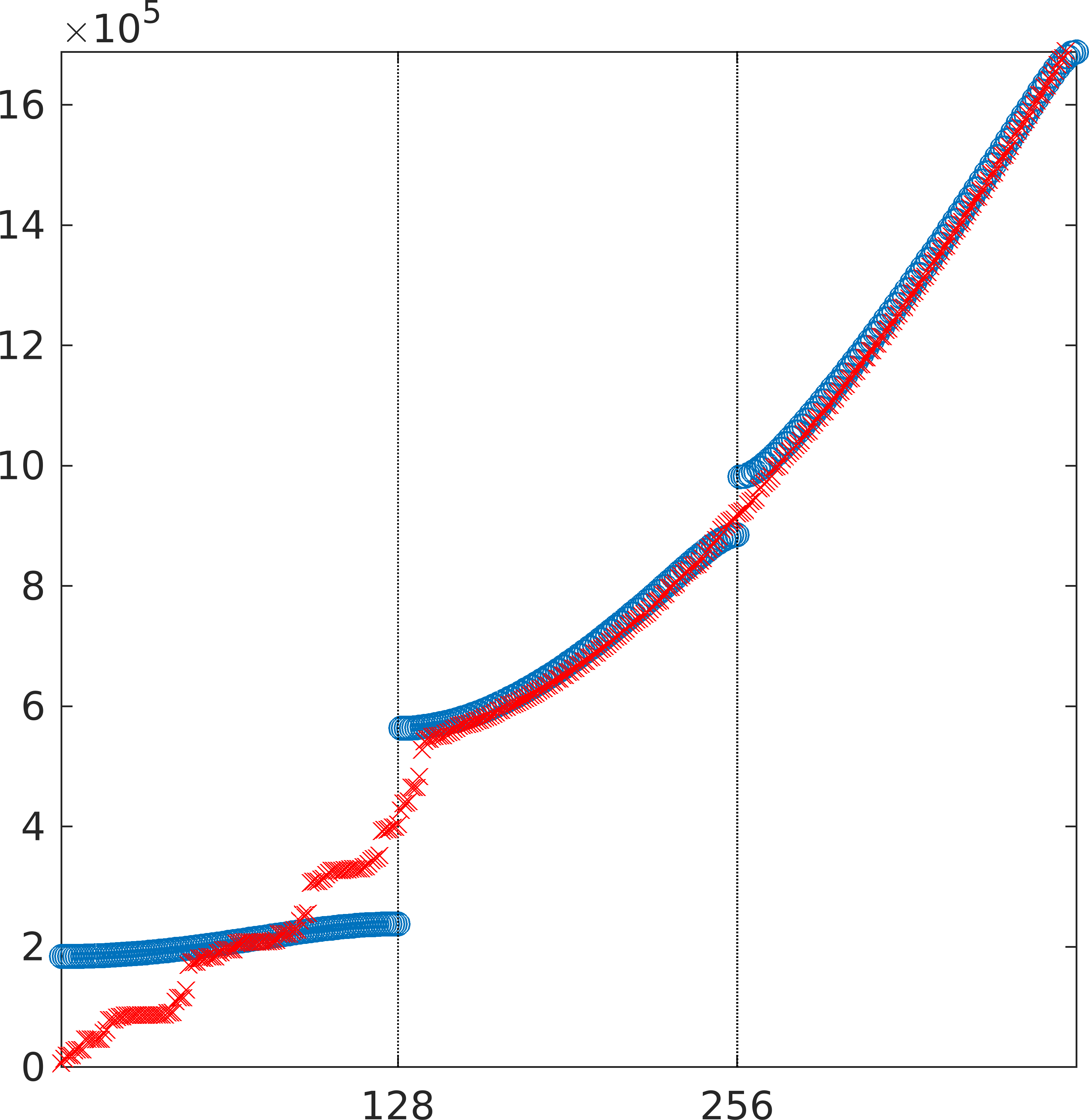}\hspace{2ex}
	\includegraphics[width=0.4\linewidth]{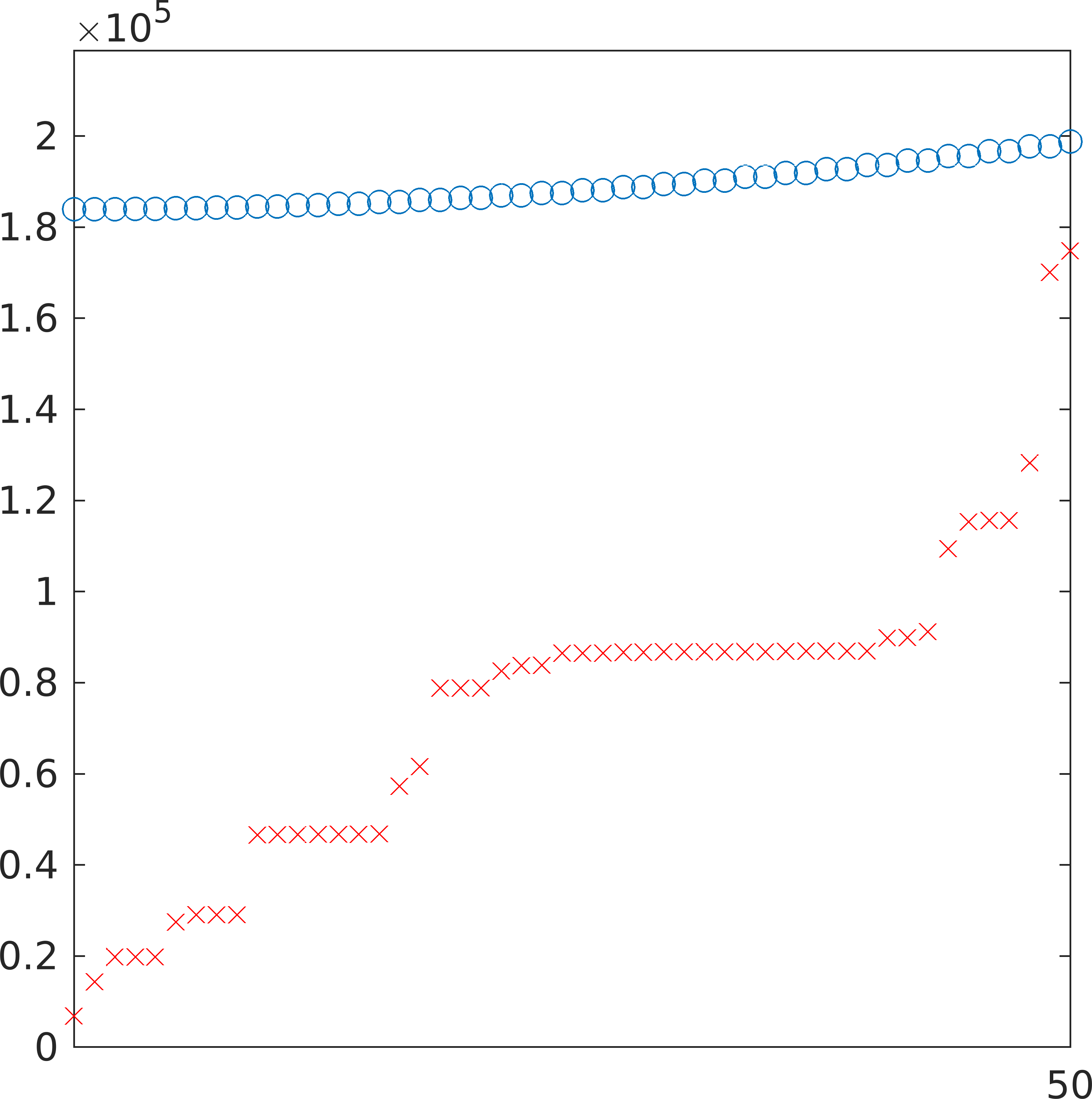}
	\caption{Schr\"odinger eigenstates in periodic and disorder potential on the unit interval. Top left: Eigenstates associated to the four smallest eigenvalues in a discontinuous periodic potential oscillating on a partition of width $\varepsilon=2^{-8}$ with values $\alpha =1$ and $\beta = 8/\varepsilon^2$, i.e.~$V(x)=\beta$ if $\lfloor{x/\eps}\rfloor$ is even and~$V(x)=\alpha$, otherwise. Top right: Eigenstates for smallest eigenvalues in a realization of a discontinuous random potential on a partition of width $\varepsilon=2^{-8}$ oscillating randomly between i.i.d.~values~$\alpha =1$ and $\beta = 8/\varepsilon^2$. Bottom row: Spectra for periodic ({\color{blue}$\circ$}) and random (${\color{red}\times}$) potential.}\label{fig:spectra}
\end{figure}

In the second step, the decay property of the Green's function is transferred to the decay of eigenstates in the following sense. There exists a subspace of dimension $K$ that contains the lowermost eigenstates up to an arbitrary accuracy \quotes{$\tol$}. More precisely, the subspace is spanned by functions with support in local sub-domains with a diameter of order $\calO(\eps\log(1/\eps)^p\log(1/\tol))$ for some exponent~$p$ 
and, hence, the eigenfunctions are well approximated by functions supported in the union of $K$ small sub-domains. This is shown in Section~\ref{sect_localization} by designing a preconditioned block inverse iteration for the solution of the eigenvalue problem. The size of the subspace has to be chosen sufficiently large so that the rate of convergence of the block inverse iteration is only weakly-depending on $\varepsilon$ (through a factor of order $\mathcal{O}(\log(1/\eps))$.

The final step then regards the estimation of the parameter $K$ which determines whether or not the localization phenomenon can be observed in the bounded domain $D$. E.g., for perfectly periodic potentials, eigenvalues are clustered in a staircase fashion with large clusters so that $K$ is of order $\eps^{-d}$ and the local sub-domains essentially aggregate to the whole domain (see Figure~\ref{fig:spectra} for an illustration and Section~\ref{sect:gaps:periodic} for the details). Here, the presence of disorder changes the picture.
In the one-dimensional model problem of Figure~\ref{fig:spectra} significant spectral gaps are observed after a few modes so that a moderate $K$ is possible. To prove that $K$ is indeed smaller in the disordered case, we study two model scenarios in Sections~\ref{sect:gaps:random} and~\ref{sect:gaps:domino}. In the first model the potential has a tensor product structure and in the second model the potential consists of randomly structured \quotes{domino blocks}. 
We show that in both cases that $K$ is of moderate size (with high probability) and, hence, the eigenstates are essentially localized to $K$ balls of radius $\eps$ each (up to logarithmic factors). 
\smallskip 

This is not the first attempt to understand this localization phenomenon mathematically. In the remaining part of the introduction, we shall give a brief survey on what is known on the localization of eigenfunctions to the (continuous) Schr\"odinger operator.

A classical localization result for non-negative, real-valued, smooth potentials on $\mathbb{R}^d$ with sufficiently high amplitude states that eigenstates below a certain energy level are exponentially localized towards infinity, cf.~\cite[Th.~3.4]{HiS96} and \cite{Agm82}. The speed of the exponential decay can be measured in an Agmon metric and depends on $V$ and the energy $E$. As for our result, localization may also be triggered by disorder. For certain types of randomly perturbed periodic potentials $V$ in full space (excluding perfectly periodic potentials), the lowermost eigenvalues of~$\calH$ are proved to have finite multiplicity and the corresponding eigenfunctions are exponentially localized towards infinity \cite[Cor.~1.4]{GeK13}. In contrast to our result, this is an asymptotic result where the rate of decay at infinity is qualitatively described in an abstract manner.

An even earlier and one of the first results in the context of localization was obtained in a one-dimensional setting. The results says that for a certain class of random potentials in $1d$ that are generated by regular Markov diffusion processes, {\it all} eigenfunctions in the spectrum decay exponentially \cite{GoM76,GMP77}. This classical result, however, cannot be generalized to higher dimensions, even for potentials with large amplitude \cite[Sect.~4.7.5]{ChS14}. Further literature on Anderson localization for random Schr\"odinger operators in full space includes \cite{CKM87,GMR15,KLS90,KMP86,KlM06}, \cite[Sect.~7.2]{GreN13}, and the references therein. There also exists a vast literature for lattice Schr\"odinger operators (also known as discrete Schr\"odinger operators), which in particular includes Anderson's original tight binding model. Here, upper bounds for the energy are typically not necessary to prove localization, provided that the disorder is sufficiently strong. 
Here we refer to the early works~\cite{FrS83,FMS85} and~\cite{AiM93,Aiz94} as well as the monograph~\cite{ChS14} for a more recent overview on localization results for discrete Schr\"odinger operators.

A recent observation in the direction of quantitative results beyond asymptotics at infinity links the localization of the ground state $u_1$ with $\| u_1 \|_{L^{\infty}(D)}=1$ on bounded domains $D\subset \mathbb{R}^d$ to the solution $\psi \in H^1_0(D)$ of the homogeneous elliptic equation $\calH\psi=1$, cf.~\cite{FilM12}. 
This so-called landscape function $\psi$ majorizes the modulus $|u_1|$ pointwise up to the multiplicative factor $E$, being the energy of $u_1$. This bound implies that in regions where $\psi$ is small compared to $E$, $u_1$ needs to be small as well. Follow-up work \cite{ArnDJMF16,ArnDFJM19b} demonstrates that the original problem can be reformulated as an eigenvalue problem with an  effective confining potential of the form $1/\psi$. In this setting, it is possible to apply the techniques of~\cite{Agm82,HiS96} to establish an exponential decay of the eigenfunction of the form $|u_1(x)| \le \operatorname{exp}(-\rho(x_0,x))$, where $x_0$ is a center of localization of the eigenfunction and $\rho(x_0,x)$ is the distance of $x$ and $x_0$ in an Agmon metric, i.e., $\rho$ minimizes the path energy $\int_{\gamma}\sqrt{}( 1/\psi(\gamma(s)) - E )_+ \ds$ amongst all paths $\gamma$ from $x_0$ to $x$. For smooth potentials, this observation shows that the eigenfunction $u_1$ needs to change at least by the factor $2$ in some ball around $x_0$ where the (a priori unknown) difference between potential and eigenvalue becomes sufficiently large~\cite{Steinerberger2017}. Since it is not known where the landscape function $\psi$ is strictly smaller than $1/E$, the above estimate may degenerate to $|u_1(x)|\le 1$ and hence, allows no rigorous a priori prediction of the localization of~$u_1$. 
Nevertheless, the landscape function was successfully applied to obtain empirically accurate predictions of localization regions~\cite{LuS18} and its local maxima provide rough approximations of the lower-most eigenvalues~\cite{ArnDFJM19}. 

In contrast to the landscape techniques which are purely a posteriori, the new quantitative results of this paper allows one to rigorously predict the emergence of exponentially localized states depending on the degree of disorder a priori.
%
\section{Schr\"odinger Eigenvalue Problem}\label{sec:evp}
This section introduces the Schr\"odinger eigenvalue problem and discusses, for a representative class of oscillatory potentials described by suitable geometric parameters, some elementary properties such as a lower bound for the minimal energy.

Throughout this paper, we use the notion $a \lesssim b$ for the existence of a generic constant $c>0$, independent of the parameters that characterize the admissible class of the potentials~$V$, such that $a \le cb$. Moreover, we use the notion $\plog$ for polynomials in the logarithm.  
%
\subsection{Model problem}\label{sec:evp:model}
We consider the eigenvalue problem of Schr\"odinger type with a highly oscillatory potential, which may reflect disorder. The following class of potentials is representative for the localization effects to be studied in this paper while its characterization by a small number of geometric and statistical parameters simplifies the presentation significantly. 
Let $\calT$ denote a partition that divides $\mathbb{R}^d$ into closed cubes with side length $\eps>0$, where $\eps^{-1} \in \mathbb{N}$ and $\eps \mathbb{Z}^d$ is the set of vertices. The partition induces a mesh on the unit cube $D := (0,1)^d$ through the quotient space $\calT \slash_{\simZ}$
with the equivalence relation for $q_1,q_2 \in \calT$ given by 
\begin{align*}
 q_1 \simZ q_2 \qquad \Leftrightarrow \qquad q_1 = \mathbf{k} + q_2\ \mbox{ for some } \mathbf{k} \in \mathbb{Z}^d.
\end{align*}
Observe that the partition $\calT \slash_{\simZ}$ consists of equivalence classes $[q]_{\simZ}$, each with exactly one representative in $D$. This definition reflects that we can consider our problem on the unit cube, extended by periodicity to the whole $\R^d$.

Defining the space of $D$-periodic $H^1$-functions by $\V := \Hper$, the corresponding variational formulation of the eigenvalue problem reads as follows: Given a non-negative potential $0\leq V\in L^\infty(D)$, find non-trivial eigenpairs $(u, \E) \in \V\times\R$ such that 
\begin{align}
\label{eq:EVP}
  a(u, v) 
  := \int_{D} \nabla u(x) \cdot \nabla v(x) + V(x)\, u(x) v(x) \dx 
  = \E\, (u, v) 
\end{align}
for all test functions $v \in \V$. Here, $(\,\cdot\,,\cdot\,)$ denotes the $L^2$-inner product on $D$. The periodic boundary conditions encoded in this variational problem are not essential and may be replaced by homogeneous Dirichlet boundary conditions. The potential~$V$ is assumed to be piecewise constant with respect to the mesh $\calT \slash_\simZ$, cf.~Figure~\ref{fig_potentials}. This prototype of large-amplitude and highly oscillatory potentials is defined through 
\[
V(x) = \begin{cases}
\alpha, & x\in \Oa, \\
\beta,  & x\in \Ob.  
\end{cases}
\]
This means that $\Oa$ and $\Ob$ are the sub-domains of $D$, on which $V$ equals $\alpha$ and $\beta$, respectively. Further, we assume $\overline{D} = \overline{\Oa} \cup \overline{\Ob}$. The corresponding subpartitions are denoted by ${\calTa\slash_{\simZ}}$ and ${\calTb\slash_{\simZ}}$.
%
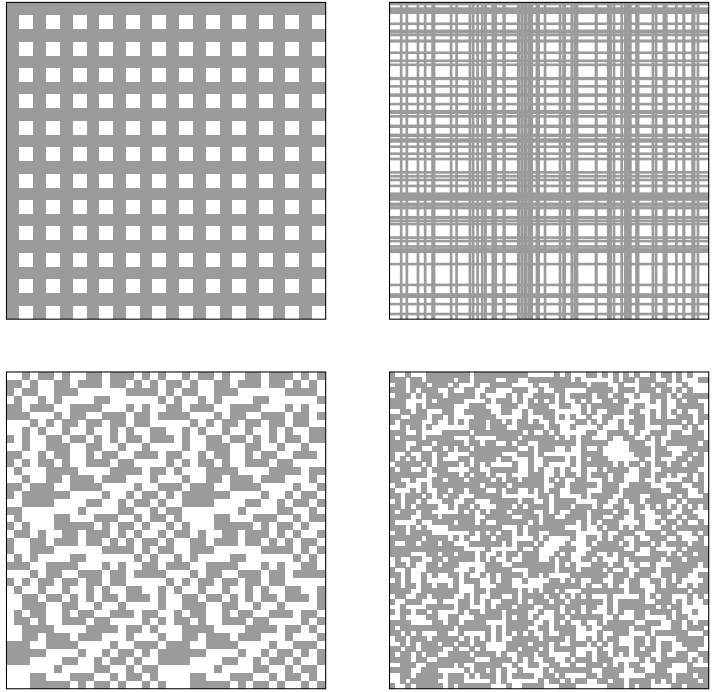
\begin{figure}
\begin{center}
\pgfmathsetseed{\number\pdfrandomseed}
\begin{tikzpicture}[scale=0.7]
%
\fill[llgray] (0.5, 1.0) -- (6.5, 1.0) -- (6.5, 7.0) -- (0.5, 7.0) -- cycle;
\foreach \i in {1, 1.5, 2, 2.5, 3, 3.5, 4, 4.5, 5, 5.5, 6, 6.5} {
	\foreach \j in {1, 1.5, 2, 2.5, 3, 3.5, 4, 4.5, 5, 5.5, 6, 6.5} {
		\fill[white] (\i,\j) -- (\i-0.25,\j) -- (\i-0.25,\j+0.25) -- (\i,\j+0.25);	
	}
}
%
%
\foreach \x in {1, ..., 65} {
	\pgfmathparse{rnd}
	\pgfmathsetmacro{\r}{\pgfmathresult}
	\draw[llgray, xshift=7.2cm, line width = 1] (\r*6+0.5, 1) -- (\r*6+0.5, 7);
}
\foreach \y in {1, ..., 65} {
	\pgfmathparse{rnd}
	\pgfmathsetmacro{\r}{\pgfmathresult}
	\draw[llgray, xshift=7.2cm, line width = 1] (0.5, \r*6+1) -- (6.5, \r*6+1);
}
%
%
\foreach \x in {0.5, 0.6, ..., 6.5} {
\foreach \y in {1, 1.1, ..., 7.0} {
    \pgfmathparse{rnd}
    \pgfmathsetmacro{\r}{\pgfmathresult}
    \ifthenelse{\lengthtest{\r pt < 0.55 pt}}{\fill[llgray, thick, yshift=-7cm, xshift=7.2cm] (\x, \y) rectangle +(0.1, 0.1);}{}
}
}
%
%
\begin{scope}[transform canvas={scale=6/40, xshift=2.3cm, yshift=-28cm}]
\foreach \p in {(0,3), 
	(0,4), 
	(0,5),
	(1,3),  
	(1,4),
	(1,5),   
	(1,8),   
	(1,9),
	(2,3),  
	(2,4),
	(2,5),
	(2,6),   
	(2,7),     
	(2,8),   
	(3,0),   
	(3,1),
	(3,4),
	(3,6),
	(3,7),
	(3,9),
	(4,0),
	(4,1),     
	(4,4),
	(4,7),
	(4,8),
	(4,9),
	(5,0),        
	(5,5),
	(5,6), 
	(5,7), 
	(6,0),
	(6,2),
	(6,5),
	(6,6),
	(7,0),
	(7,3),
	(7,5),
	(7,6),
	(7,8),
	(7,9),
	(8,1),
	(8,3),
	(8,5),
	(8,8),
	(9,0),
	(9,1),
	(9,4),
	(9,6),
	(9,8),
	(9,9),
	(2,10),
	(2,11),
	(3,10),
	(3,11),
	(4,10),
	(5,11),
	(7,10),
	(8,10),
	(6,11),
	(8,11),
	(0,13),
	(1,12),
	(2,13),
	(4,12),
	(5,13),
	(7,12),
	(8,13),
	(10,12),
	(0,15),
	(1,15),
	(3,14),
	(4,13),
	(6,14),
	(7,13),
	(10,13),
	(0,17),
	(0,18),
	(2,15),
	(5,14),
	(1,16),
	(1,19),
	(3,16),
	(3,17),
	(3,18),
	(2,19),
	(5,15),
	(4,17),
	(5,17),
	(4,18),
	(4,19),
	(6,15),
	(7,15),
	(6,18),
	(7,18),
	(7,19),
	(8,15),
	(9,14),
	(8,16),
	(9,16),
	(8,18),
	(8,19),
	(10,11),
	(11,10),
	(10,14),
	(11,14),
	(10,18),
	(10,19),
	(11,18),
	(11,19),
	(12,11),
	(13,11),
	(14,10),
	(15,11),
	(16,11),
	(11,13),
	(12,12),
	(13,13),
	(14,13),
	(15,12),
	(16,12),
	(15,12),
	(12,15),
	(13,14),
	(14,14),
	(14,15),
	(16,14),
	(12,17),
	(13,17),
	(14,17),
	(15,17),
	(16,16),
	(12,19),
	(13,18),
	(14,19),
	(15,18),
	(16,19),
	(10,1),
	(11,0),
	(12,0),
	(13,1),
	(14,0),
	(15,1),
	(16,1),
	(11,2),
	(11,3),
	(12,2),
	(13,3),
	(15,2),
	(10,6),
	(11,6),
	(10,7),
	(11,7),
	(10,9),
	(11,8),
	(12,4),
	(13,5),
	(12,7),
	(13,7),
	(15,6),
	(15,7),
	(14,9),
	(15,9),
	(13,9),
	(12,8),
	(14,3),
	(16,3),
	(14,4),
	(15,4),
	(16,5),
	(16,7)} 
{
	\fill[llgray] \p rectangle +(1, 1);
	\fill[llgray,xshift=20cm] \p rectangle +(1, 1);
	\fill[llgray,yshift=20cm,xshift=3cm] \p rectangle +(1, 1);
	\fill[llgray,yshift=20cm,xshift=23cm] \p rectangle +(1, 1);
}
\foreach \p in { 
	(17,4),
	(17,5),
	(17,17),
	(17,14),
	(17,15),
	(17,11),
	(17,18),
	(17,0),
	(17,2),
	(17,8),
	(17,9),
	(19,10),
	(19,11),
	(19,14),
	(19,0),
	(18,1),
	(18,2),
	(19,4),
	(19,3),
	(18,18),
	(18,16),
	(19,17),
	(19,19),
	(18,5),
	(19,6),
	(18,14),
	(18,15),
	(19,12),
	(18,7),
	(19,8),
	(18,9)} 
{
	\fill[llgray] \p rectangle +(1, 1);
	\fill[llgray,xshift=20cm] \p rectangle +(1, 1);
	\fill[llgray,yshift=20cm,xshift=3cm] \p rectangle +(1, 1);
	\fill[llgray,yshift=20cm,xshift=-17cm] \p rectangle +(1, 1);
}
\end{scope}

%
\draw (0.5, 1.0) -- (6.5, 1.0) -- (6.5, 7.0) -- (0.5, 7.0) -- cycle;
\draw[xshift=7.2cm] (0.5, 1.0) -- (6.5, 1.0) -- (6.5, 7.0) -- (0.5, 7.0) -- cycle;
\draw[yshift=-7cm] (0.5, 1.0) -- (6.5, 1.0) -- (6.5, 7.0) -- (0.5, 7.0) -- cycle;
\draw[xshift=7.2cm, yshift=-7cm] (0.5, 1.0) -- (6.5, 1.0) -- (6.5, 7.0) -- (0.5, 7.0) -- cycle;
\end{tikzpicture} 
\end{center}
\caption{Illustrations of the potential $V$ in two space dimensions, in which the gray parts represent $\Ob$ (where $V(x)=\beta$) and the white parts $\Oa$ (where $V(x)=\alpha$). Periodic (upper left), realization of a random tensor product (upper right), domino block (lower left), and a fully random potential (lower right).} 
\label{fig_potentials}
\end{figure}
We are interested in the particular regime of $\beta\gg 1$, moderate $\alpha\ge0$, and small $\eps$. Furthermore, we assume that $\beta$ is not smaller than $\eps^{-2}$, which we will make more precise later on. We have in mind potentials where the distribution of $\Oa$ and $\Ob$ follows statistical laws. However, the actual statistics will become relevant only in Section~\ref{sect:gaps} in connection with the identification of spectral gaps.

In Section~\ref{sect_localization} we will exploit the operator formulation of \eqref{eq:EVP}. For this, we introduce the operators $\calA\colon \V\to\V^*$ and $\calI\colon \V\to\V^*$, defined by
\[
  \langle \calA u, v \rangle_{\V^*,\V} := a(u, v), \qquad
  \langle \calI u, v \rangle_{\V^*,\V} := (u, v)
\]  
for functions $u,v\in\V$. Note that $\calA$ denotes the weak form of the Schr\"odinger operator $\calH$ and that the eigenvalue problem \eqref{eq:EVP} is equivalent to $\calA u = E\, \calI u$. 

To shorten notation, we simply write $\Vert\cdot\Vert:=\sqrt{(\cdot,\cdot)}$ for the canonical $L^2$-norm on $D$. We also introduce the $V$-weighted $L^2$-norm,
\[
  \Vert v \Vert_V^2 
  := (Vv,v)
  = \int_D V(x)\,  |v(x)|^2\dx 
\]
as well as the energy norm,
\[
\Vvert v \Vvert^2
:= a(v, v)
= \Vert \nabla v\Vert^2 + \Vert v\Vert_V^2.
\]
Furthermore, we denote the norm on a sub-domain $\Oa$ or $\Ob$ by an additional subscript. 
%
\subsection{Geometry and cut-off function}\label{sec:evp:cutoff}
Before we can estimate the Schr\"odinger states from below, we need to introduce additional notation on the geometry of the potential. First, we define the set of cubes within the partition $\calT \slash_{\simZ}$, namely
\[
  \calQ 
  := \left\{ \hspace{2pt} \bigcup_{i=1}^{m} \hspace{2pt} [q_i]_{\simZ} \ \Big|\ Q = \bigcup_{i=1}^{m} q_i \subseteq \R^d \text{ is a (closed) cube and union of $m$ elements } q_i\in \calT \right\}.
\]
Note that this implies that all cubes in $\calQ$ have side length $\eps k$ for some natural number $1 \le k\le \eps^{-1}$. Second, we define the set of {\em maximal cubes} in $\Oa$ and $\Ob$, respectively, by 
\[
  \calQnu 
  := \big\{ Q\in \calQ\ |\ Q \subseteq \overline{\Onu} \text{ and there is no } Q' \in \calQ \text{ with } Q\subset Q'\subseteq \overline{\Onu} \big\},
\]
for $\nu=\alpha,\beta$. Note that $\bigcup_{Q\in\calQnu}Q = \overline{\Onu}$. Since $\calT \slash_{\simZ}$ is a quotient space, we can interpret $\calQ$ and $\calQnu$ as containing \quotes{cubes} that are extended over the periodicity interface. Such cubes are connected in $\R^d$, but can be disconnected as subsets of the unit cube $D$. Whenever one of the following arguments requires an element of $\calQ$ or $\calQnu$ to be connected, we shall interpret it as a subset of $\R^d$, where values outside of $D$ are obtained through periodicity. This will be done without further  mentioning. For brevity, we shall from now on abuse the notation and simply write $\calT$ instead of $\calT \slash_{\simZ}$. The same is done for $\calTa$ and $\calTb$.

The cubes in $\calQa$ somehow characterize the potential valleys, i.e., regions where the potential has the value $\alpha$. Finally, we define the {\em maximal width of a potential valley} in $\calT$ by 
\[
  L := \max_{Q\in \calQa} \frac{h_Q}{\eps} \in \N,
\]
where $h_Q$ denotes the side length of a cube $Q$. 
In the trivial setting $V\equiv \beta$, where $\calQa$ is empty, we set $L:=1$.  
With this characteristic value, we are able to bound the maximum number of overlaying maximal cubes in $\calQa$, namely
\[
  \kappa_\calT
  :=  \max_{q\in \calTa} \big| \{ Q\in\calQa\ |\ q\subseteq Q \} \big|
  \le L^d.
\]
\begin{remark}
In the periodic setup as in Figure~\ref{fig_potentials} (upper left) we have $L=1$ and $\kappa_\calT=1$. Note that the value of $L$ remains unchanged if $\alpha$ and $\beta$ are swapped in this example. 
\end{remark}
\begin{remark}
In the one-dimensional setting there are no overlapping maximal cubes, i.e., we have $\kappa_\calT = 1$ for $d=1$. 
\end{remark}
The exponential localization of the Green's function and eigenstates in oscillatory potentials requires sufficiently high amplitudes of the potential. This is quantified in the subsequent assumption depending on the oscillation length $\eps$. Loosely speaking we shall assume that the strength of potential peaks, characterized by the parameter $\beta$, is large compared to the inverse of the square of the oscillation length $\eps$. This assumption resembles the scaling in a typical physical setup \cite{NatureGreinerEtAl}. Considering for instance an optical lattice potential $V$ that is based on a laser diode operating at a wavelength $\eps=\lambda$, then the potential oscillates at a frequency of order $\eps^{-1}$. On the other hand, the maximum potential depth $\beta$ is measured in units of the recoil energy, which itself is proportional to $\lambda^{-2}=\eps^{-2}$. This is precisely the relation that we shall assume for $\eps$ and $\beta$. 
Assumptions on the strength of the potential valleys, characterized by the parameter $\alpha$, are not needed for the exponential decay of the Green's function. 
Thus, until Section~\ref{sect:gaps}	we only assume $0\le\alpha\le\beta$, which includes the particular case of the constant potential $V\equiv \beta$. 
\begin{assumption}
\label{ass_epsBeta}
The coefficient $\beta$ is large in the sense that it satisfies the estimate $\beta \gtrsim \eps^{-2}$ and we assume that $\calTb\neq\emptyset$. 
\end{assumption}
To derive energy estimates we introduce a cut-off function $\eta\colon D \to [0,1]$. This function is assumed to be smooth, constant $1$ in $\Oa$, and vanishes in each cube of side length $\eps/2$, which is centered in elements of $\calTb$, cf.~Figure~\ref{fig_cutoff}. In other words, $\eta$ hits zero in each $\beta$-peak of the potential $V$. Further, we assume that $\Vert \nabla\eta\Vert_{L^\infty(D)} \lesssim \eps^{-1}$. We emphasize that this implies some kind of Friedrichs inequality. 
%
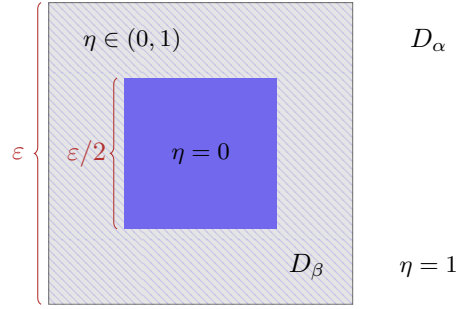
\begin{figure}
\begin{center}
	\begin{tikzpicture}[scale=1.0]
	\fill[gray, opacity=0.2] (0,0) -- (4,0) -- (4,4) -- (0,4) -- cycle;
	\draw[thin, pattern=north west lines, pattern color=myBlue, opacity=0.5] (0,0) -- (4,0) -- (4,4) -- (0,4) -- cycle; 
	\fill[myBlue] (1,1) -- (3,1) -- (3,3) -- (1,3) -- cycle;
	\node at (3.4, 0.5) {$\Ob$};	
	\node at (5.0, 3.5) {$\Oa$};	
	\draw[myRed,decorate,decoration={brace}] (-0.1, 0) -- (-0.1, 4);
	\node[myRed] at (-0.4, 2) {$\eps$};	
	\draw[myRed,decorate,decoration={brace}] (0.9, 1) -- (0.9, 3);
	\node[myRed] at (0.5, 2) {$\eps / 2$};	
	\node at (5.0, 0.5) {\small $\eta = 1$};	
	\node at (1.1, 3.5) {\small $\eta \in (0,1)$};
	\node at (2.0, 2.0) {\small $\eta = 0$};
	\end{tikzpicture}
\end{center}
\caption{Illustration of the cut-off function $\eta$, which is constant $1$ in~$\Oa$ and vanishes in the interior of each element of $\calTb$. }
\label{fig_cutoff}
\end{figure}
\begin{lemma}
\label{lem_eta_Friedrich}
Assume that $\calTb\neq\emptyset$. 
Consider a function $v\in \V$ and the cut-off function $\eta$ introduced above. The product $\eta v$ then satisfies the estimate 
$$
\| \eta v \|_{L^2(D)} \leq
\cl  \hspace{1pt} \eps L \hspace{2pt} \|\nabla (\eta v) \|_{L^2(D)}
$$
with some generic constant $\cl\lesssim \kappa_\calT L^{d}\le L^{2d}$.
\end{lemma}
The proof of the lemma is given in Appendix \ref{appendix-a}, where we also present a refined version of this Friedrichs-type inequality. 
\begin{example}
In the extreme case of only a single $\beta$-element in $\calT$, we have $L\approx \eps^{-1}$ and the result degenerates in the sense that we loose the factor~$\eps$ in the estimate. 	
On the other hand, if the potential satisfies $V\equiv \beta$, then we get the classical Friedrichs inequality~$\| \eta v \|_{L^2(D)} \le \eps\, \|\nabla (\eta v) \|_{L^2(D)}$. 
In the case of a random potential with correlation length of the order $\varepsilon$, one obtains with high probability maximal valleys of size $L\approx\log(1/\eps)$. Thus, we get a Friedrichs-type inequality, which contains the factor~$\eps$ but also logarithmic terms. 
\end{example}
%
\subsection{Lower bound on the energy}\label{sec:evp:bound}
With the estimate of Lemma~\ref{lem_eta_Friedrich}, we are able to give a lower bound for the spectrum of $\calH$. 
For this, we will bound the scaled energy of a function $v\in\V\setminus\{ 0\}$, 
\begin{align}
\label{def-Ev}
\E(v) 
:= \frac{a(v,v)}{\Vert v\Vert^2}
= \frac{\Vvert v\Vvert^2}{\Vert v\Vert^2}
\end{align}
from below. In the assumed regime $\beta\gtrsim\varepsilon^{-2}$ this lower bound is in the order of~$\eps^{-2}$. 
In the following, we will no longer mention the silent convention that $E(v)$ is only defined for $v\not=0$.
\begin{lemma}
Under Assumption~\ref{ass_epsBeta} we have 
\[
E^1 
:= \min_{v\in \V} \E(v)
\gtrsim \frac{1}{\cl^2 (\eps L)^2}.
\]
\end{lemma}
\begin{proof}
For an arbitrary function $v\in \V$ we obtain 
\begin{align*}
  \Vert v \Vert^2
  &= \Vert v \Vert_\Oa^2 + \frac{1}{\beta}\, \Vert v \Vert_{V,\Ob}^2  \\
  &\le \Vert \eta v \Vert^2 + \frac{1}{\beta}\, \Vert v \Vert_{V,\Ob}^2 \\
  &\lesssim 
  \cl^2 \hspace{1pt} (\eps L)^2\, \Vert \nabla(\eta v) \Vert^2 + \frac{1}{\beta}\, \Vert v \Vert_{V, \Ob}^2 \\
  &\lesssim \cl^2 \hspace{1pt} (\eps L)^2\,
  \Big( \frac{1}{\eps^2\beta}\, \Vert v \Vert_{V,\Ob}^2  + \Vert \nabla v \Vert^2 \Big) + \frac{1}{\beta}\, \Vert v \Vert_{V, \Ob}^2 \\
  &= \cl^2 \hspace{1pt} (\eps L)^2\, \Vert \nabla v \Vert^2 + 
  \big( 1 + \cl^2 L^{2}  \big)  \frac{1}{\beta}\, \Vert v \Vert_{V, \Ob}^2.
\end{align*}
Thus, Assumption~\ref{ass_epsBeta} yields the estimate 
\begin{align}
\label{estimate_L2norm}
  \Vert v \Vert^2
  \lesssim \cl^2 \hspace{1pt} (\eps L)^2\, \Vvert v \Vvert^2.
\end{align}
This shows that the energy is bounded from below by
\begin{align*}
\E(v) 
= \frac{\Vvert v\Vvert^2}{\Vert v\Vert^2}
\gtrsim \frac{1}{\cl^2 (\eps L)^2}.
\end{align*}
Using the characterization of eigenvalues by the Rayleigh quotient, we directly obtain the stated lower bound for the ground state of the Schr\"odinger equation.
\end{proof}
\begin{remark}
\label{rem_sharp}
Estimate~\eqref{estimate_L2norm} is sharp with respect to the maximum width of potential valleys $\eps L$, i.e., there exists a non-trivial function $u\in \V$ with $\Vert u \Vert^2 \gtrsim (\eps L)^2 \, \Vvert u \Vvert^2$. To see this we consider the first eigenfunction $u\in H^1_0(Q_L)$ of the shifted Laplace eigenvalue problem 
\[ 
  \int_{Q_L} \nabla u(x) \cdot \nabla v(x) \dx 
  = \big(\lambda-\alpha\big) \int_{Q_L} u(x)\, v(x) \dx 
\]	
with test functions $v\in H^1_0(Q_L)$ on a maximal $\alpha$-cube $Q_L \in \calQa$ with side length $\eps L$. According to \cite[Ch.~10.4]{Str08}, the first eigenvalue satisfies $\lambda-\alpha = d \hspace{2pt} \pi^2 / (\eps L)^2$ such that~$u$, extended by zero to a function in $\V$, satisfies 
\[
  \Vert u\Vert^2
  = \Vert u\Vert_{Q_L}^2
  = \frac{1}{\lambda}\, \Vvert u\Vvert_{Q_L}^2
  = \frac{1}{\lambda}\, \Vvert u\Vvert^2
  \gtrsim \frac{(\eps L)^2}{d\pi^2} \Vvert u\Vvert^2,
\]
provided that~$\alpha$ is of moderate size. 
Similar arguments will be used in Section~\ref{sect:gaps} where we prove the existence of spectral gaps. 
\end{remark}
%
%
\section{Exponential Decay of the Green's Function}\label{sec:precond}
This section shows that the Green's function associated with the Schr\"odinger operator decays exponentially relative to the parameter $\varepsilon$ that reflects the characteristic length of oscillation of the potential. The proof is strongly inspired by a recent innovative proof of the exponential decay of the corrector Green's function in the context of numerical homogenization for arbitrarily rough diffusion coefficients by Kornhuber and Yserentant \cite{KorY16}, see also \cite{KorPY18} and earlier work \cite{MalP14, HenP13, Pet16}. The idea is to show that the Schr\"odinger operator can be preconditioned by an operator that is local with respect to a decomposition of the domain into cubic sub-domains with diameter~$2\eps$. 
The spectrum of the arising preconditioned operator is proved to be clustered around~$1$ so that simple iterative solvers approximate the action of the inverse Schr\"odinger operator applied to some compactly supported function (or the point evaluation functional) up to an accuracy $\tol$ in only~$\mathcal{O}(\plog(1/\eps) \log(1/\tol))$ steps. 

The locality of the preconditioned operator ensures that the diameter of the support of the approximation is of the same order. This means that the Green's function associated with $\calH$ decays exponentially in units of~$\eps L$. The result is independent of the degree of disorder of the potential and remains valid in the perfectly periodic case.
%
\subsection{Overlapping domain decomposition on the $\varepsilon$-scale}
We introduce an overlapping decomposition of $D$, which we will later use to define the local preconditioner. 
For this, we consider the nodes corresponding to the mesh $\calT$, which we denote by~$\calN$. 
For each node $z\in \calN$ let $\lambda_z$ be the standard $Q_1$-basis function \cite[Sect.~3.5]{BreS08}, i.e., $\lambda_z$ is a piecewise polynomial of partial degree one with $\lambda_z(z)=1$ and $\lambda_z(w)=0$ for any other node $w\in\calN\setminus\{z\}$. Again one has to take care of the assumed periodicity of the domain and the resulting identification of the boundary nodes. All together, this gives a set of functions, which forms a partition of unity on $D$, i.e., 
\begin{align}
\label{eqn_partunity}
\sum\nolimits_{z\in \calN} \lambda_z \equiv 1.
\end{align}
The patches of the $Q_1$-hat functions define small subdomains 
\[
  D_z := \supp \lambda_z
\]
for each $z\in\calN$. 
By definition, $D_z$ are cubes of side length~$2\eps$ and each $T\in\calT$ is contained in $2^d$ of these subdomains.  
For an illustration of such a patch we refer to Figure~\ref{fig_localPatch}. 
\begin{figure}
\begin{center}
\begin{tikzpicture}[scale=0.9]
	\foreach \point in {(0,4), 
						(1,5), (1,6), 
					    (2,4), (2,6), 
					    (4,4), (4,5),
					    (5,5), (5,6), 
					    (6,4), (6,5), 
					    (7,4), (7,6)} {
		\fill[llgray] \point -- +(0, 1) -- +(1, 1) -- +(1, 0);
	}
	\foreach \x in {1, 5, 7} {		
		\fill[llgray] (\x, 3.5) -- (\x, 4) -- (\x+1, 4) -- (\x+1, 3.5) -- cycle;
	}
	\foreach \x in {0, 1, 2, 5, 6} {		
		\fill[llgray] (\x, 7.5) -- (\x, 7) -- (\x+1, 7) -- (\x+1, 7.5) -- cycle;
	}	
	\foreach \y in {5, 6} {		
		\fill[llgray] (0, \y) -- (0, \y+1) -- (-0.5, \y+1) -- (-0.5, \y) -- cycle;
	}
	\fill[pattern=north west lines, pattern color=red, opacity=0.5] (1, 7) -- (3, 7) -- (3, 5) -- (1, 5) -- cycle;
	\fill[pattern=north east lines, pattern color=blue, opacity=0.4] (2, 6) -- (4, 6) -- (4, 4) -- (2, 4) -- cycle;
	\fill[pattern=north east lines, pattern color=green!70!black, opacity=0.7] (6, 7.49) -- (8, 7.49) -- (8, 6) -- (6, 6) -- cycle;
	%
	\draw[ultra thick] (0, 3.5) -- (0, 7) -- (8.5, 7);
	\foreach \y in {4, 5, 6, 7} {
		\draw[thin] (-0.5, \y) -- (8.5, \y);
	}
	\foreach \x in {0, 1, 2, 3, 4, 5, 6, 7, 8} {
		\draw[thin] (\x, 3.5) -- (\x, 7.5);
	}
	\fill[red] (2,6) circle (0.3em);	
	\fill[blue] (3,5) circle (0.3em);	
	\fill[green!65!black] (7,7) circle (0.3em);	
\end{tikzpicture} 
\end{center}
\caption{Illustration of local (overlapping) patches $D_i$ for three exemplary nodes. Gray squares are regions with $V(x)= \beta$.
}
\label{fig_localPatch}
\end{figure}
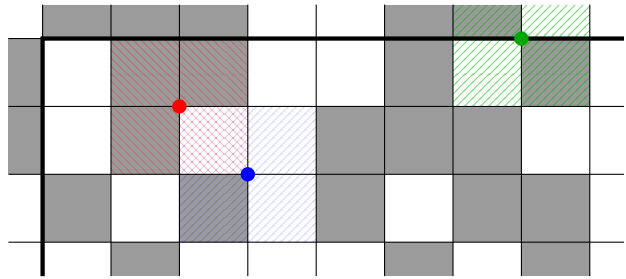

Based on the patches $D_z$, we define local $H^1$-spaces by
\[
  \V_z 
  := H^1_0(D_z)
  = \big\{ v\in H^1(D_z)\ |\ v=0 \text{ on } \partial D_z  \big\}.
\]
Elements of $\V_z$ are considered to be extended by zero outside of $D_z$. Moreover, recall that we interpret $D_z$ as a subset of $\R^d$. This implies that a function $v\in H^1_0(D_z)$ respects the periodicity over opposite edges (or faces for $d=3$) of $\partial D$ and does not necessarily fulfill $v=0$ on $\partial D$. 
We define the corresponding $a$-orthogonal projection $\Pz\colon \V\to\V_z$ by the variational problem
\[
  a(\Pz u, v) = a(u,v)
\]
for test functions $v\in \V_z$. Note that the trivial embedding $\V_z\hook\V$ allows to consider $\Pz$ as a mapping from $\V$ to $\V$. We may also define $\tPz\colon \V^*\to \V_z$ by
\[
  a(\tPz F, v) = \langle F, v\rangle_{\V^*,\V} 
\]
for test functions $v\in \V_z$. Letting $\calA\colon \V\to\V^*$ denote the operator representation of $a(\cdot\,,\cdot)$, we have the relation $\Pz = \tPz \calA$.  
\subsection{Optimal $\varepsilon$-local preconditioner}
Combining all local projections, we obtain the operator 
\begin{align}
\label{def_calP}
  \calP 
  := \sum\nolimits_{z\in\calN} \Pz.
\end{align}
This defines a mapping $\calP\colon \V \to\V$ if we assume that the canonical embeddings $\V_z\hook \V$ are exploited. It is easy to see that this operator is continuous. Accordingly, we define $\tilde \calP\colon \V^*\to\V$ by~$\tilde \calP:= \sum\nolimits_{z\in\calN} \tPz = \calP \calA^{-1}$. 
We emphasize that the operator $\calP$ is quasi-local with respect to the $\eps$-mesh $\calT$, since \quotes{information} can only propagate distances of order $\eps$ each time that $\calP$ is applied. 

The remaining part of this section aims to show that $\tilde\calP$ defines a good approximation of $\calA^{-1}$ and thus, serves well as a preconditioner within iterative solvers for linear equations and the Schr\"odinger eigenvalue problem. Following the abstract theory for additive subspace correction
or additive Schwarz methods for operator equations \cite{KorY16} (see also \cite{Xu92,yserentant_1993} for the matrix case) we need to verify that the energy norm of a function $u\in \V$ can be bounded in terms of the sum of local contributions from $\V_Q$ and $\V_z$. 
\begin{lemma}
\label{lem_K2}
For every decomposition $u = \sum_{z\in\calN} u_z$ with $u_z\in \V_z$ we have 
\[
  \Vvert u \Vvert^2 
  \le K_2\, \sum_{z\in\calN} \Vvert u_z \Vvert^2 
\]
with $K_2 = 2^d$.
\end{lemma}
\begin{proof}
We use the local supports of $u_z$ and the fact that for $T\in\calT$ there are at most $2^d$ functions $u_z$ with support on $T$. Thus, we can estimate on a single element,  
\[
  \Vvert u \Vvert_{T}^2 
  = \bigVvert \sum\nolimits_{z\in\calN} u_z \bigVvert_{T}^2
  \le 2^d \sum\nolimits_{z\in\calN} \Vvert u_z \Vvert_T^2.
\]
A summation over all $T$ yields the assertion.
\end{proof}
We now need the reverse estimate for one specific decomposition of $u\in\V$ in the local spaces~$\V_z$. 
Therefore, we define the local functions $u_z := \lamz u$ for all $z\in\calN$. 
From~\eqref{eqn_partunity} we know that $\sum_{z\in\calN} u_z = u$. 
For this particular decomposition we can prove the following lemma. 
\begin{lemma}
\label{lem_K1}
Given Assumption~\ref{ass_epsBeta} and the decomposition of $u\in\V$ as above, it holds that
\[
  \sum_{z\in\calN} \Vvert u_z \Vvert^2 
  \lesssim K_1\, \Vvert u \Vvert^2 
\]
with constant $K_1 := 2^{d+1} (1 + \cl^2 L^2 \big)$. 
\end{lemma}
\begin{proof}
With the estimate~\eqref{estimate_L2norm} we directly obtain 
\begin{align*}
  \sum_{z\in\calN} \Vvert u_z \Vvert^2
  &= \sum_{z\in\calN}\Big(\Vert \nabla(\lamz u) \Vert^2 + \Vert \lamz u \Vert_V^2 \Big)\\
  &\le \sum_{z\in\calN} \Big( 2\, \Vert \nabla u \Vert_{D_z}^2 + 2\eps^{-2}\Vert u \Vert_{D_z}^2 + \Vert u \Vert_{V,D_z}^2 \Big)\\
  &\le 2^{d+1} \Vvert u \Vvert^2 + 2^{d+1} \eps^{-2} \Vert u \Vert^2\\  
  &\lesssim 2^{d+1} (1 + \cl^2 L^2) \Vvert u \Vvert^2. 
\end{align*}
Note that we have again used the fact that the maximal number of overlapping patches is $2^d$.
\end{proof}
\begin{remark}
In the periodic setting with $L=1$ one can show that $K_1 \lesssim 2^{d+2}$, i.e., $K_1$ is independent of $\eps$.
\end{remark}
Note that we have used Assumption~\ref{ass_epsBeta} in the previous lemma. We emphasize that such a condition is necessary, since the general case would lead to a constant $K_1 \approx \eps^{-2}$ in the worst case.  
The subsequent result is a direct consequence of the previous Lemmata~\ref{lem_K2} and \ref{lem_K1}, cf.~\cite[Lem.~3.1 and Th.~3.2]{KorY16}.
\begin{corollary}
\label{cor_K1K2}
Given Assumption~\ref{ass_epsBeta}, we obtain the norm equivalence 
\begin{align}
\label{estimates_P}
  K_1^{-1} a(v,v) 
  \le a(\calP v, v)
  \le K_2\, a(v,v) 
\end{align}
for all $v\in\V$ with the constants $K_1$ and $K_2$ from Lemmata~\ref{lem_K1} and~\ref{lem_K2}. 
\end{corollary}	
Recall that $\calP\colon\V\to\V$ from \eqref{def_calP} has led to the definition of $\tilde\calP\colon\V^*\to\V$ by $\tilde{\calP}\calA = \calP$. With this operator, the estimate \eqref{estimates_P} can be rewritten in the form 
\[
  {K_1}^{-1}\, a(v,v) 
  \le \langle \calA \tilde\calP \calA v, v \rangle
  \le K_2\, a(v,v). 
\]
We summarize a number of properties of the operator $\calA \tilde\calP \calA$.
\begin{lemma}
\label{lem_APA}
The operator $\calA \tilde\calP \calA\colon\V \to\V^*$ is symmetric, coercive, and continuous. As a consequence, the operator is also invertible. 
\end{lemma}
\begin{proof}
The symmetry follows from the definition of $\calP$, namely
\begin{align*}
  \langle \calA \tilde\calP \calA u, v \rangle 
  = a(\calP u, v)
  &= \sum\nolimits_{Q\in\calQa} a(\PQ u, v) + \sum\nolimits_{z\in\calNb} a(\Pz u, v) \\
  &= \sum\nolimits_{Q\in\calQa} a(u, \PQ v) + \sum\nolimits_{z\in\calNb} a( u, \Pz v)
  = a(u, \calP v)
  = \langle u, \calA \tilde\calP \calA v \rangle .
\end{align*}
The coercivity follows directly from the coercivity of $a(\cdot\,,\cdot)$ and \eqref{estimates_P}. Finally, the continuity follows from the boundedness of $a(\cdot\,,\cdot)$ and $\calP$.  
\end{proof}
Estimates of the form \eqref{estimates_P} are well-known from the preconditioner community for the computation of eigenvalues of a symmetric and positive definite matrix. The following approximation result, together with the local computability, then results in a well-designed preconditioner. 
\begin{theorem}
\label{thm_tildeP_gamma} 
Under Assumption~\ref{ass_epsBeta} and with the scaling factor $\vartheta := 1/(K_2+K_1^{-1})$ with $K_1$ and $K_2$ from Lemmata~\ref{lem_K2} and~\ref{lem_K1}, there exists a positive constant $\gamma_\calP < 1$ such that  
\[
  \Vvert \id - \vartheta {\calP}\Vvert
  = \Vvert \id - \vartheta\tilde{\calP}\calA\Vvert 
  := \sup_{v\in\V} \frac{\Vvert v - \vartheta\tilde{\calP}\calA v\Vvert}{\Vvert v\Vvert}
  \le \gamma_\calP < 1.
\]
\end{theorem}	
\begin{proof}
By the spectral equivalence \eqref{estimates_P} we conclude that for any $v\in\V$ it holds 
\[
  (1 - \vartheta K_2)\, \Vvert v \Vvert^2 
  \le a(v-\vartheta {\calP}v, v)
  = a(v,v) - \vartheta a(\calP v, v)
 \le (1- \vartheta K_1^{-1})\, \Vvert v \Vvert^2. 
\]
Furthermore, for any linear operator $\calQ\colon \V\to\V$ we have
\[
  \Vvert \calQ\Vvert^2
  := \sup_{v\in\V, \Vvert v\Vvert = 1} a(\calQ v, \calQ v) 
  \le \sup_{v\in\V, \Vvert v\Vvert = 1} \sup_{w\in\V, \Vvert w\Vvert = 1} a(\calQ v, w)\, \Vvert \calQ\Vvert
  \le \Vvert \calQ\Vvert^2.
\]
Thus, all estimates are in fact equalities. 
If $a(\calQ\, \cdot\,, \cdot)$ defines in addition a scalar product in $\V$, then we get by the polarization identity 
\[
  \Vvert \calQ\Vvert 
  = \sup_{v\in\V, \Vvert v\Vvert = 1} \sup_{w\in\V, \Vvert w\Vvert = 1} a(\calQ v, w)
  = \sup_{v\in\V, \Vvert v\Vvert = 1} a(\calQ v, v). 
\]
By the mentioned spectral equivalence we know that $\calQ := \id -\vartheta {\calP}$ defines a scalar product for $\vartheta < 1/K_2$ such that the choice $\vartheta := 1/(K_2+K_1^{-1})$ gives 
\[
  \Vvert \id - \vartheta {\calP}\Vvert
  = \sup_{v\in\V, \Vvert v\Vvert = 1} a((\id - \vartheta {\calP})v, v)
  \le 1- \vartheta K_1^{-1}  
  = \frac{K_2}{K_1^{-1} + K_2}. \qedhere
\]
\end{proof}	
The proof of Theorem~\ref{thm_tildeP_gamma} provides an explicit formula of the upper bound, namely $\gamma_\calP \le K_2 / (K_1^{-1} + K_2)$.  
This shows that, given Assumption~\ref{ass_epsBeta}, $\gamma_\calP$ only depends on the geometry of the potential, which is encoded in the constants $K_1$ and $K_2$, but not on the actual values $\alpha$ and $\beta$ and thus, not on their contrast.  
Note, however, that  $\gamma_\calP$ depends on $L$, which itself may depend on $\eps$. 
\subsection{Exponential decay}
\label{sect_decay_Greens_func}
In the last part of this section we want to relate the previous results to the exponential decay of the Green's function associated with the differential operator $\calH$. For this, let $f\in L^2(D)$ be a given function with local support and consider the problem of finding $u \in \V$ with $\calA u=F$, where $F:=(f,\,\cdot\,)\in \V^*$. To approximate $u$, we define the iteration 
\begin{align}
\label{sec3-preliminiary-iterations}
  u^{(k)} 
  := u^{(k-1)} + \vartheta \big( \tilde \calP F - \calP u^{(k-1)} \big)
  = u^{(k-1)} + \vartheta \calP \big( u - u^{(k-1)} \big)
\end{align}
for $k\ge 1$ and trivial starting value $ u^{(0)}=0$. 
First, we observe that $u^{(1)}=\vartheta \tilde \calP F = \vartheta \calP u$ is a local function, because its computation only includes the solution of local problems. To see this, note that 
\[
	u^{(1)} 
	= \vartheta\tilde\calP F
	= \sum_{z \in\calN} \vartheta\tilde\calP_z F
	=: \sum_{z \in\calN} u_z^{(1)} ,
\]
where all $u_z^{(1)}$ are fully defined by local test functions $v_z\in\V_z$ through 
\[ 
	a(u_z^{(1)} , v_z) 
	= \vartheta\, a(\calA^{-1} F, v_z)    
	= \vartheta\, \langle F, v_z \rangle_{\V^*, \V} 
	= (\vartheta f, v_z).
\]
This implies that the application of $\tilde \calP$ maintains locality in the sense that the support of $u^{(1)}=\vartheta\tilde\calP F$ is at most one~$\eps$-layer larger than the support of $f$. Inductively, we immediately see that the support of 
$$
  u^{(k)} = u^{(1)}  + (\id - \vartheta \calP )\, u^{(k-1)} 
$$
is at most $k$ $\eps$-layers larger than the support of $f$.  
Next, we observe from the definition of $u^{(k)} $ in~\eqref{sec3-preliminiary-iterations} that 
\begin{align*}
  u - u^{(k)} 
  = (\id - \vartheta \calP) \big( u - u^{(k-1)} \big)
  = (\id - \vartheta \calP)^k \big( u - u^{(0)} \big)
  = (\id - \vartheta \calP)^k u.
\end{align*}
Applying Theorem \ref{thm_tildeP_gamma}, we obtain
\begin{align*}
  \Vvert u - u^{(k)} \Vvert 
  \le \gamma_\calP^k\, \Vvert u \Vvert,
\end{align*}
i.e., we have that $u^{(k)}$ converges exponentially fast to $u$ with rate~$\gamma_\calP$. 
Recall the earlier observation that the support of $u^{(k)}$ is at most $k$ $\eps$-layers bigger than the the support of the local function $f$, i.e., $\supp(u^{(k)})\subseteq B^\infty_{k\eps }(\supp f)$, where $B^\infty_{r}$ denotes the ball of radius $r$ in the sup norm. 
This then shows that the Green's function associated with the Schr\"odinger operator must have an exponential decay as summarized in the following corollary. 
\begin{corollary}[Exponential decay of the Green's function]
Consider Assumption~\ref{ass_epsBeta} and let $f\in L^2(D)$ be local. 
Then, the solution $u\in \calV$ of the variational problem $\calA u = f$ decays exponentially fast in the sense of
\[  
  \Vvert u\Vvert_{D\setminus B^\infty_{k\eps}(\supp f)}
  \lesssim \gamma_\calP^k\, \Vvert u\Vvert.
\]
Here, $0<\gamma_\calP<1$ denotes the constant from Theorem~\ref{thm_tildeP_gamma}.
\end{corollary} 
\begin{example}
We again consider the extreme cases and analyze the number of necessary steps to achieve $\gamma_\calP^k \le \tol$. 
In the worst case~$L\approx \eps^{-1}$, i.e., $K_1\approx 2^{d+1}\eps^{-2}$, we need $k = \mathcal{O}(\log(1/\tol)\, \eps^{-2})$ steps. 
Thus, the solution~$u$ may not be localized. 
On the other hand, the periodic setting with $L=1$ yields $K_1 = 2^{d+2}$ and thus, $k = \mathcal{O}(\log(1/\tol))$. This means that the number of needed steps to reach the error tolerance is independent of $\eps$. 
In the case of interest with $L\approx\log(1/\eps)$ we have~$k = \mathcal{O}(\log(1/\tol)\log(1/\eps)^p)$ for a certain polynomial degree $p\le 4d+2$. 
This means that the number of steps only depends logarithmically on $\eps$ and that $u$ is of local nature.  
\end{example}
The obtained decay result is in agreement with the well-known exponential decay of the Green's function associated with $\calH$ for positive $V$ on a sufficiently large sub-domain. In the case of constant potentials $V\equiv \beta$, this is for instance shown in \cite[Lem.~3.2]{Glo11}. 
We shall revisit the discussion of this paragraph later in Section~\ref{sect_localization_oneStep} as part of the localization proof for the eigenfunctions of $\calH$.
\begin{remark}
\label{rem_smallGamma}
For later arguments, it is important to achieve an error reduction factor (per step) below some prescribed value~$\gap<1$. 
This is easily achieved by considering multiple steps of the iteration~\eqref{sec3-preliminiary-iterations} with preconditioner $\vartheta\calP$. 
More precisely, we introduce the operator~$\calPP\colon \V\to\V$, which includes~$k$ steps with $k$ large enough such that  
\[
	\Vvert \id - \calPP \Vvert
	\le \gamma:= \gamma^k_\calP 
	\le \frac{1-\gap}{2} < 1.
\]
This then implies $\gamma+\gap \le (1+\gap)/2 < 1$. 
It goes without saying that this enlarges the support of the solution, i.e., the operator~$\calPP$ spreads information over $k$ $\eps$-layers. 
The order of $k$ can be estimated by 
\[
	k \approx \frac{\log(1-\gap) - \log 2}{\log \gamma_\calP}.
\]  
In the discussed case $L\approx\log(1/\eps)$, where $\gamma_\calP$ depends polylogarithmically on~$\eps$, we have $k = \calO(\log(1/(1-\gap)) \log(\plog(1/\eps)))$.
\end{remark}
%
%
\section{Quantitative Localization of Eigenfunctions}\label{sect_localization}
In the following we want to transfer the localization arguments from Section \ref{sect_decay_Greens_func} to the spectrum of $\calH$. 
Assuming a sufficiently large spectral gap between the smallest and the ($K\hspace{-0.5ex}+\hspace{-0.5ex}1$)-st eigenvalue for some moderate $K$, we show that the ground state is indeed quasi-local. The assumption will later turn out to be valid for potentials with a high level of disorder. Nevertheless, the convergence results of this section are independent of the actual value of $K$ in the sense that the ground state is always in the span of $K$ exponentially localized functions. This shows localization of the eigenfunction itself if $K$ is sufficiently small compared to $\eps^{-d}$. 
%
\subsection{Inverse power iteration}\label{sect_localization_power}
As mentioned in Section~\ref{sec:evp:model}, the Schr\"odinger eigenvalue problem \eqref{eq:EVP} can be written as an operator equation in the dual space of $\V$, namely  
\[
  \calA u = E\, \calI u. 
\]
Recall that $\calA\colon\V\to\V^*$ is the operator corresponding to $a(\cdot\,,\cdot)$ whereas $\calI\colon\V\to\V^*$ denotes the extension of the inner product in $L^2(D)$. We consider the inverse power method for PDE eigenvalue problems and illustrate the method first for the case of a spectral gap after the first eigenvalue $E^1$. Due to the ellipticity of $\calA$, we may assume that there exists a Hilbert basis of $\V$ composed of the eigenfunctions $u_1, u_2, \dots$ normalized in the $L^2(D)$-norm. Further, we assume that a given starting function $v^{(0)}\in \V$ satisfies $(u_1, v^{(0)})\neq 0$, i.e., we may express $v^{(0)}$ in the form $v^{(0)} = \sum_{i=1}^\infty \alpha_i u_i$ with $\alpha_1\neq 0$.  

The inverse power method, known from numerical linear algebra \cite[Ch.~10.3]{AllK08}, also converges in the Hilbert space setting, cf.~\cite{EriSL95, AltF18ppt}. The iteration, including a normalization by $E^1$, has the form 
\begin{align}
\label{eqn_powIteration}
  v^{(k)} 
  = {E^1} \calA^{-1} \calI v^{(k-1)}
  =: \calB v^{(k-1)}
  = \calB^k v^{(0)}
\end{align}
with $\calB := E^1\, \calA^{-1} \calI\colon \V\to\V$. 
With $\gap:= E^1/E^2 < 1$ the iteration leads to 
\[
  v^{(k)} 
  = \alpha_1 u_1 + \sum_{i=2}^\infty \alpha_i\, \Big( \frac{E^1}{E^i}\Big)^k u_i.    
\]
Note that~$(u_1, v^{(k)}) = \alpha_1$ remains unchanged due to the scaling factor $E^1$. 
Measuring the distance of $v^{(k)}$ to the eigenspace of $u_1$ in the energy norm by
\[
  \err^{(k)} 
  := \min_{c\in \R} \Vvert v^{(k)} - c\,u_1 \Vvert
  = \Vvert v^{(k)} - \alpha_1 u_1 \Vvert
  = \Big( \sum_{i=2}^\infty |\alpha_i|^2 \Big( \frac{E^1}{E^i}\Big)^{2k} E^i \Big)^{1/2} ,  
\]
we obtain due to the orthogonality of the eigenfunctions, 
\[
  \err^{(k)} 
  = \Vvert v^{(k)} - \alpha_1 u_1 \Vvert 
  \le \gap^k \Vvert v^{(0)} - \alpha_1 u_1\Vvert  
  = \gap^k \err^{(0)}. 
\] 
This shows the exponential convergence of the inverse power iteration with a rate depending on the gap between the first two eigenvalues. 	
%
%
\subsection{Preconditioned iteration step}\label{sect_localization_oneStep}
For the proof of localization of the Schr\"odinger states we need to replace the inverse iteration \eqref{eqn_powIteration} by an inexact iteration including the preconditioner from Section~\ref{sec:precond}. 
More precisely, we like to replace one inverse power step by a fixed number of preconditioned steps, which we indicate by the operator~$\calPP$, cf.~Remark~\ref{rem_smallGamma}. 
Recall that this includes an error reduction by a factor~$\gamma$ with $\gamma+\gap < 1$ and that $\calPP$ enlarges the support by only a few~$\eps$-layers. 
One step of the inverse power iteration $v^{(k)} = \calB v^{(k-1)}$ is replaced by 
\begin{align}
\label{ideal_PINVIT}
	\tilde v^{(k)}
	:= \tilde v^{(k-1)} + \calPP \big( E^1 \calA^{-1}\calI \tilde v^{(k-1)} - \tilde v^{(k-1)} \big) 
	= \tilde v^{(k-1)} + \calPP \big( \calB \tilde v^{(k-1)} - \tilde v^{(k-1)} \big).
\end{align}
In the numerical linear algebra community, this iteration is called {\em preconditioned inverse iteration} (PINVIT) \cite{DyaO80,BraKP95,Kny98} if the (unknown) factor $E^1$ is replaced by an approximation of the energy, e.g., by the Rayleigh quotient. 
Here, however, we consider again the exact scaling by the first eigenvalue. 
\begin{remark}
For practical computations it is also of interest that $\tilde v^{(k)}$ is cheap to compute. This is indeed the case, since its computation only involves the solution of local problems. 
\end{remark}
%
The locality of the iterates was already discussed in Section~\ref{sect_decay_Greens_func}, i.e., up to logarithmic terms in $1/(1-\gap)$ and $1/\eps$ the support of $\tilde v^{(k)}$ is at most $k$ $\eps$-layers larger than the support of the initial function $v^{(0)}$. 
It remains to show the exponential convergence of the iteration scheme is maintained, despite the inclusion of the preconditioner. For this, we show that the error reduces by a fixed factor in every iteration step.  

As before, we assume $v^{(0)} = \sum_{i=1}^\infty \alpha_i u_i$ with $\alpha_1\neq 0$. 
For the (exact) inverse iteration we have seen that $(u_1, v^{(0)})$ remains unchanged during the iteration process. 
This changes by the implementation of the preconditioner. 
However, assuming $\tilde v^{(k-1)} = \sum_{i=1}^\infty \hat \alpha_i u_i$ we can estimate 
\begin{align*}
  \err^{(k)} 
  = \min_{c\in \R} \Vvert \tilde v^{(k)} - c\,u_1 \Vvert
  &\le \Vvert \tilde v^{(k)} - \hat\alpha_1 u_1\Vvert \\
  &\le \Vvert \calB \tilde v^{(k-1)} - \hat\alpha_1 u_1 \Vvert 
   + \Vvert (\id- \calPP) ( \calB \tilde v^{(k-1)} - \tilde v^{(k-1)} )\Vvert \\
  &\le \Vvert \sum_{i=2}^\infty \tfrac{E^1}{E^i} \hat \alpha_i u_i \Vvert 
   + \gamma\, \Vvert \sum_{i=2}^\infty (\tfrac{E^1}{E^i}-1) \hat\alpha_i u_i  \Vvert \\
  &\le \gap\, \Vvert \tilde v^{(k-1)} - \hat\alpha_1 u_1\Vvert 
   + \gamma\, \Vvert \tilde v^{(k-1)} - \hat\alpha_1 u_1\Vvert \\
  &= (\gap + \gamma)\, \err^{(k-1)}.
\end{align*} 
Thus, we have an error reduction by a factor $(\gap + \gamma)$ in each step. 
%
%
\subsection{Block iteration}\label{sect_localization_block}
Since we cannot guarantee a spectral gap after the first eigenvalue, we need a block iteration. 
For the general case we assume that there is a spectral gap within the first $K+1$ eigenvalues, which leads to the definition $\gap:= E^1/E^{K+1} < 1$. We aim to perform a block version of the inverse power iteration \eqref{eqn_powIteration}. For this we need a starting subspace to initiate the power iteration. 
%
Let $V^{(0)}$ denote a basis of such a $K$-dimensional subspace, 
\[
  V^{(0)}
  = \begin{bmatrix}
   v^{(0)}_1, & v^{(0)}_2, & \dots, & v^{(0)}_K
  \end{bmatrix} 
  \in \V^K.
\]
As before, we can express these functions in terms of the eigenfunctions of the Schr\"odinger operator, i.e., for $j=1, \dots, K$,  
\begin{align*}
  v_j^{(0)} = \sum_{i=1}^\infty \alpha_{i,j}\, u_i 
  \qquad\text{and}\qquad
  \calC:= [\alpha_{i,j}]_{i,j= 1, \dots, K} \in \R^{K,K}. 
\end{align*}
The matrix $\calC$ contains the coefficients of the initial functions $v_j^{(0)}$ in terms of $u_1, \dots, u_{K}$. As generalization of the condition $\alpha_1\neq 0$ in Section~\ref{sect_localization_power}, we assume here that~$\calC$ is invertible. 
%
The block inverse iteration (or simultaneous inverse iteration) including normalization then reads
\begin{align}
\label{eqn_blockIteration}
  V^{(k)}
  = {E^1} \calA^{-1} \calI V^{(k-1)}
  = \calB^k V^{(0)}.
\end{align}
For a single function $v_j^{(0)}$ this means 
\[
  v_j^{(k)}  
  = \calB^k v_j^{(0)}
  = {(E^1)^k} \sum_{i=1}^\infty \alpha_{i,j} (\calA^{-1}\calI)^k u_i
  = \sum_{i=1}^\infty \alpha_{i,j} \Big( \frac{E^1}{E^i}\Big)^k u_i.  
\]
With $x := \calC^{-1} e_1 \in \R^{K}$ the linear combination $V^{(k)}x \in \V$ satisfies 
\[
  V^{(k)}x
  = \sum_{i=1}^\infty\, [\alpha_{i,1}, \dots, \alpha_{i,K}]\, x\, \Big( \frac{E^1}{E^i}\Big)^k u_i  
  = u_1 + \sum_{i=K+1}^\infty \overline\alpha_{i} \Big( \frac{E^1}{E^i}\Big)^k u_i,  
\]
where $\overline\alpha_i := [\alpha_{i,1}, \dots, \alpha_{i,K}]\, x$. 
Similarly as in Section~\ref{sect_localization_power} we show that~$V^{(k)}x$ converges exponentially to the span of~$u_1$ with rate~$\gap$,
\[
  \err^{(k)} 
  = \min_{c\in\R} \Vvert V^{(k)}x - c\, u_1 \Vvert 
  = \Vvert V^{(k)}x - u_1 \Vvert
  \le \gap^k \Vvert V^{(0)}x - u_1 \Vvert
  = \gap^k \err^{(0)}.
\]
Note that the initial error $\err^{(0)}$ is bounded in terms of $\calC^{-1}$ and the energy of the starting functions $v^{(0)}_1, v^{(0)}_2, \dots, v^{(0)}_K$. 
Thus, for the convergence of the block iteration it remains to find a suitable starting block $V^{(0)}$. Its precise construction depends on the considered potential $V$, see e.g.~Section~\ref{sect:gaps:random:starting} for a potential with a tensor product structure or 
~Section~\ref{sect:gaps:domino:starting} for a potential consisting of domino blocks of different size. 
To prove the quasi-locality of the Schr\"odinger ground state, we need to include the preconditioner from Section~\ref{sec:precond}. 
%
%
\subsection{Inexact block iteration}\label{sect_localization_inexactIter}
We install the preconditioner $\calPP$ into the block iteration, which yields a block version of the preconditioned inverse iteration introduced in Section~\ref{sect_localization_oneStep}. 
Given $V^{(0)}$ we locally compute a sequence $\tilde V^{(k)}$ by a simultaneous application of the preconditioned iteration. 
In the main result of this paper we show that the first eigenfunction $u_1$ is essentially an element of $\sspan \tilde V^{(k)}$, i.e., a $K$-dimensional function space that is spanned by basis functions that are exponentially decaying in distances of $\eps$. 
If $V^{(0)}$ only contains local functions and $K$ is of moderate size, i.e., if there is a significant spectral gap after the first few eigenvalues, then $u_1$ itself is exponentially localized as the linear combination of $K$ exponentially localized functions. 
\begin{theorem}[Convergence of inexact block iteration]
\label{theorem-convergence-inexact-block}
Given Assumption~\ref{ass_epsBeta}, $\gap = E^1/E^{K+1}$, and a prescribed tolerance $\tol$, we consider a starting subspace~$V^{(0)}$ with invertible coefficient matrix~$\calC$. 
Assume that the preconditioner~$\calPP$ from Remark~\ref{rem_smallGamma} satisfies $\gamma \lesssim \gap^k$ with $k \approx \log(1/\tol) / \log(1/\gap)$. 
Then, $k$ steps of the preconditioned block iteration yields an approximation $\tilde v \in \sspan \tilde V^{(k)}$ with		
\[
\Vvert \tilde v - u_1 \Vvert 
\lesssim \tol\, \err^{(0)}
= \tol\, \Vvert V^{(0)} \calC^{-1} e_1 - u_1\Vvert. 
\]  
Moreover, the support of $\tilde v$ is only $k^2$ $\eps$-layers larger than the union of the supports of the starting functions in $V^{(0)}$.
\end{theorem}
\begin{proof}
First note that, due to the scaling by $E^1$, we have $\Vvert \calB\Vvert \le 1$. 
We prove that the inexact iteration yields a good approximation of the inverse power method. 
For this, we compare the space obtained by the inexact block iteration $\tilde V^{(k)}$ with $V^{(k)}$. 
Since we consider a simultaneous iteration, it is sufficient to consider a single vector of~$V^{(0)}$, which we denote by~$v^{(0)} \in \V$.  
For the error $e^{(k)} := v^{(k)} - \tilde v^{(k)}$ we have 
\begin{align*}
  e^{(k)} 
  = v^{(k)} - \tilde v^{(k)} 
  &= \calB v^{(k-1)} - \tilde v^{(k-1)} - \calPP \big(\calB \tilde v^{(k-1)} - \tilde v^{(k-1)} \big) \\
  &= (\id - \calPP) \big( \calB v^{(k-1)} - \tilde v^{(k-1)} \big) + \calPP \calB e^{(k-1)} \\
  &= (\id - \calPP) (\calB - \id) v^{(k-1)} + (\id - \calPP) e^{(k-1)} + \calPP \calB e^{(k-1)}.   
\end{align*}
Thus, with $v^{(k-1)} = \calB^{k-1} v^{(0)}$, $\Vvert\calB\Vvert\le 1$, and $\Vvert\calB-\id\Vvert\le 2$, we obtain  
\begin{align*}
  \Vvert e^{(k)}\Vvert  
  \le 2 \gamma\, \Vvert v^{(k-1)} \Vvert + (1+2\gamma)\, \Vvert e^{(k-1)} \Vvert 
  \le 2 \gamma\, \Vvert v^{(0)} \Vvert + (1+2\gamma)\, \Vvert e^{(k-1)} \Vvert.
\end{align*}
%
Since $e^{(0)} = v^{(0)} - v^{(0)} = 0$, we conclude 
\[
  \Vvert e^{(k)} \Vvert  
  \le 2\gamma\, \Vvert v^{(0)} \Vvert\, \sum_{\nu=0}^{k-1} (1+2\gamma)^\nu 
  \lesssim \gamma\, (1+2\gamma)^{k} \Vvert v^{(0)} \Vvert.   
\]
From Section~\ref{sect_localization_block} we know that for~$k \approx \log(1/\tol) / \log(1/\gap)$ steps of the inverse power method we have~$\Vvert V^{(k)}x - u_1 \Vvert \lesssim \gap^k \err^{(0)}$. 
Thus, with~$\gamma \lesssim \gap^k \approx \tol$ we conclude by the triangle inequality 
\[
  \Vvert \tilde V^{(k)} x - u_1 \Vvert
  \lesssim \Vvert e^{(k)} \Vvert + \Vvert V^{(k)}x - u_1 \Vvert 
  \lesssim \tol \err^{(0)}. 
\qedhere
\]
\end{proof}
Theorem~\ref{theorem-convergence-inexact-block} shows that the choice of $V^{(0)}$ is crucial for the localization result. 
First, its locality bounds the support of the constructed approximation of $u_1$. 
Second, the quality of the starting subspace enters the estimate through~$\err^{(0)}$, which can be bounded in terms of the matrix $\calC^{-1}$ and the energies of $V^{(0)}$, namely by 
\[
  \err^{(0)}
  = \Vvert V^{(0)} x - u_1 \Vvert 
  = \Vvert V^{(0)} \calC^{-1} e_1 - u_1 \Vvert 
  \lesssim \|\calC^{-1}\|_1 \max_{j=1,\dots,K}  \Vvert v_j^{(0)} \Vvert.
\]
The verification of the locality of $V^{(0)}$ and the boundedness of $\|\calC^{-1}\|_1$ depends on the considered potential and will be in the focus of Section~\ref{sect:gaps}. 
\begin{remark}
Note that the localization result of Theorem~\ref{theorem-convergence-inexact-block} is not optimal in the sense that the support of $\tilde V^{(k)}$ grows quadratically in $k$. 
To improve this, we need to show that a fixed number of preconditioner steps is sufficient as outlined in Remark~\ref{rem_smallGamma}. 
In a similar fashion as in Section~\ref{sect_localization_oneStep} we can show that for $\tilde V^{(k-1)}$ with corresponding coefficient matrix $\hat \calC\in \R^{K,K}$, $\hat x:= \hat \calC^{-1}e_1$, we have
\begin{align*}
	\min_{y\in\R^K} \Vvert \tilde V^{(k)}y - u_1 \Vvert
	\le \Vvert \tilde V^{(k)}\hat x - u_1 \Vvert 
	&\le (\gap + \gamma)\, \Vvert \tilde V^{(k-1)} \hat x - u_1 \Vvert. 
\end{align*}
Thus, we have a similar error reduction as in the non-block case if we can prove that~$\tilde V^{(k-1)} \hat x$ is at least a quasi-optimal approixmation of $u_1$ in the span of $\tilde V^{(k-1)}$. 
This, in turn, would imply that the ground state $u_1\in \calV$ decays exponentially fast in the sense of 
\[  
\Vvert u_1\Vvert_{D\setminus B^\infty_{k\eps}(\operatorname{supp} V^{(0)})}
\lesssim (\gap + \gamma)^{k}\, \err^{(0)}. 
\]
\end{remark}
\begin{remark}
\label{remark-on-theorem-convergence-inexact-block}
The result of Theorem~\ref{theorem-convergence-inexact-block} generalizes to the first $r$ eigenfunctions provided that the gap condition $E^r/E^{K+1} + \gamma<1$ holds true. For this, one needs to consider $\calB := E^r \calA^{-1} \calI$ and $x := \calC^{-1} e_r$.
\end{remark}
%
%
\section{Application to Prototypical Potentials}\label{sect:gaps}
This section aims to validate the assumption on the spectral gap in Theorem~\ref{theorem-convergence-inexact-block} in three model scenarios. 
This will turn out to fail in the periodic case. The lower part of the spectrum indeed decomposes into well separated eigenvalue clusters, but the clusters are too large. The introduction of disorder changes the picture. We first state two general results, which are needed in this context. 
\begin{lemma}
\label{lem_RayleighBound}
Recall the definition of the energy $E(v)$ from \eqref{def-Ev} and let $u_1, \dots, u_N \in \V$ be orthogonal w.r.t.~$(\cdot,\cdot)$ as well as $a(\cdot,\cdot)$. Furthermore, assume the uniform bounds
\[
  c_1 \le \E(u_i) \le c_2
\]
for all $i=1,\dots, N$. Then, $c_1 \le \E(u) \le c_2$ for all $u\in\sspan \{u_1, \dots, u_N \}$.
\end{lemma}
\begin{proof}
The assumption on the Rayleigh quotient of $u_i$ can be translated into 
\[
  c_1 \Vert u_i \Vert^2 
  \le a(u_i, u_i) 
  \le c_2 \Vert u_i \Vert^2 
\]
for all $i=1,\dots, N$. For a given linear combination $u:= \sum_{i=1}^N \alpha_iu_i \in \sspan \{u_1, \dots, u_N \}$ we get due to the assumed orthogonality
\[
  \Vert u \Vert^2
  = \big\Vert \sum\nolimits_{i=1}^n \alpha_i u_i \big\Vert^2
  = \sum\nolimits_{i=1}^n \alpha_i^2\, \Vert u_i\Vert^2
  \le \frac{1}{c_1} \sum\nolimits_{i=1}^n \alpha_i^2\, a(u_i, u_i)
  = \frac{1}{c_1}\, a(u,u)
\]
and thus, $c_1\le\E(u)$. Analogously, one shows that $\E(u)\le c_2$.
\end{proof}	
We emphasize that orthogonality in both inner products is especially given for functions having disjoint support. 
\begin{lemma}
\label{lem_evBound}
Assume that $u_1, \dots, u_N\in \V$ are pairwise orthogonal w.r.t.~$(\cdot,\cdot)$ and $a(\cdot,\cdot)$. If we have the property $\E(u_i)\le c$ for all $i=1,\dots, N$ then the eigenvalue problem 
\[
  a(u,v) = E\, (u,v)
\]
for test functions $v\in\V$ has at least $N$ eigenvalues with $E \le c$.
\end{lemma}
\begin{proof}
The Courant min--max principle in Hilbert spaces states that the $\ell$-th eigenvalue satisfies
\[
  E^\ell 
  = \min_{\dim \V^{(\ell)} = \ell} 
	\hspace{4pt} 
	\max_{\ v\in \V^{(\ell)}} 
	\hspace{3pt}  \E(v). 
\]
Thus, for $\ell \le N$, the choice $\V^{(\ell)} := \sspan \{u_1, \dots, u_{\ell} \}$ yields together with the previous lemma 
\[
  E^\ell 
  \le \max\nolimits_{v\in \sspan \{u_1, \dots, u_{\ell} \}} E(v) 
  \le \max\nolimits_{v\in \sspan \{u_1, \dots, u_N \}} E(v) 
  \le c. \qedhere
\]
\end{proof}	
In the following we derive bounds for the spectral gaps of $\calH$, where we first investigate the case of periodic potentials. As we will see, in sufficiently disordered media, significant spectral gaps are expected to appear much earlier than in the periodic case, cf.~Figure~\ref{fig:spectra}. 
Note that, up to now, we have only assumed $\beta$ to be large (cf.~Assumption~\ref{ass_epsBeta}). For the proof of spectral gaps we also need $\alpha$ to be small. 
\begin{assumption}
\label{ass_epsAlpha}
The coefficient $\alpha$ is of moderate size in the sense that $\alpha \lesssim (\eps L)^{-2}$, i.e., the contrast satisfies $\beta/\alpha\gtrsim L^2$. 
\end{assumption}
%
%
\subsection{Periodic potential}\label{sect:gaps:periodic}
We aim to derive lower and upper eigenvalue bounds for the periodic case shown in Figure~\ref{fig_potentials} (upper left). This includes $N :=(2\eps)^{-d}$ cubes of side length~$\eps$, on which $V$ takes the value $\alpha$. Recall that we have $L=1$ in this case.
\subsubsection{Upper eigenvalue bounds}
We provide an upper bound on the first $\ell N$ eigenvalues by the construction of particular functions and Lemma~\ref{lem_evBound}. Restricted to a single element $q\in\calT$, on which the potential $V$ equals $\alpha$, we consider the standard Laplace eigenvalue problem with homogeneous Dirichlet boundary and shift $\alpha$. For this problem, eigenfunctions and -values are well-known \cite[Ch.~10.4]{Str08}. On $q$ the first $\ell$ eigenfunctions $\hat u_1, \dots, \hat u_\ell$ (extended by zero on $D$) satisfy the bound 
\[
E(\hat u_j) 
\le \alpha + \frac{\pi^2}{\eps^2} \big( \ell^2 + d-1 \big)
\lesssim \frac{\ell^2}{\eps^2},
\]
since $\alpha \lesssim (\eps L)^{-2}$ by Assumption~\ref{ass_epsAlpha}. As this holds true for each of the $N$ cubes in $\calQa$, Lemma~\ref{lem_evBound} yields  
\begin{align}
\label{eqn_periodic_upperBound}
  E^{N\ell} \lesssim \frac{\ell^2}{\eps^2}.
\end{align}
Note that we exploit here the orthogonality of the eigenfunctions on a single element $q$ of~$\calT$ and the orthogonality due to the disjoint support of the cubes. 
\subsubsection{Lower eigenvalue bounds}
In order to prove gaps in the spectrum, we also need lower bounds on the eigenvalues. For this, we use the reformulation of the min-max principle, namely the max-min principle. This is then combined with quasi-interpolation results from the theory of finite elements.
\begin{lemma}
\label{lem_periodic_lowBounds}
Assume the periodic setting with $L=1$ and $N=(2\eps)^{-d}$ cubes on which the potential $V$ equals $\alpha$, as before. If $\beta \gtrsim (\ell-1)^2 \eps^{-2}$, then it holds the estimate 
\begin{align}
\label{eqn_periodic_lowerBound}
  E^{N \ell^d+1} 
  \ \gtrsim\ \frac{(\ell-1)^{2}}{\eps^{2}}.
\end{align}
\end{lemma}
\begin{proof}
We apply the max-min eigenvalue characterization of the form  
\begin{align}
\label{eqn_lambda_Nell}
  E^k 
  \ =\ \adjustlimits \max_{\dim\V^{(k-1)}=k-1\ \ } \min_{v \in [\V^{(k-1)}]^\text{c}}\ E(v), 
\end{align}
where $[\V^{(k-1)}]^\text{c} \subset \V$ is any complementary space to the $(k-1)$-dimensional space $\V^{(k-1)}$.
The strategy is to construct a subspace $\W \subseteq \V$ of dimension $N \ell^d$. For this, we consider a uniform triangulation $\calT_h$ of $D$ into cubes with local mesh size $h:=\eps/(\ell-1)$. The corresponding set of nodes is denoted by $\mathcal{N}_h$. The finite element space $\W$ is then defined by the span of all $Q_1$-basis functions corresponding to the nodes in $\mathcal{N}_h \cap \overline{\Oa}$. Note that the dimension of $\W$ equals $N \ell^d$ and that functions in $\W$ may have a slight support in $\Ob$, namely one layer of width $h$ surrounding the $\alpha$-valleys. 

To characterize a complementary space we construct a local  projection operator
\[
  \Pi\colon \V = \Hper \to \W 
\]
and define $\W^\text{c} := \ker \Pi$. The operator is based on the Scott-Zhang quasi-interpolation operator $\ISZ\colon H^1(\Oa)\rightarrow \W\vert_{\Oa}$, cf.~\cite{ScoZ90, HeuS07}. The operator $\Pi$ is then defined by the property 
$$(\Pi u)\vert_{\Oa} = \ISZ (u\vert_{\Oa})$$
in a unique way. Since $\ISZ$ does not depend on the behavior of functions in $\Ob$, $\Pi$ has the important property that information is not spread from $\Oa$ to $\Ob$. By the properties of the Scott-Zhang interpolation, this implies the $a$-stability of $\Pi$ with a constant depending only on $\beta h^2$. Due to $\V= \ker \Pi \oplus \range \Pi$, $\W^\text{c}$ is indeed a  closed complementary space. Moreover, by construction we have the error estimate
\[
  \Vert u - \Pi u\Vert_{\Oa}=\Vert u - \ISZ u\Vert_{\Oa}
  \lesssim h\, \Vert \nabla u \Vert_{\Oa}.
\]
For $u \in\W^\text{c}$ we thus have by the assumption $\beta \gtrsim (\ell-1)^2 \eps^{-2}=h^{-2}$ that
\[
  \Vert u\Vert^2 
  = \Vert u - \Pi u\Vert^2_{\Oa} + \Vert u\Vert^2_\Ob  
  \lesssim h^2 \Vert \nabla u \Vert^2_{\Oa} + \frac 1\beta \Vert u\Vert^2_{V, \Ob}
  \lesssim \frac{\eps^2}{(\ell-1)^2}\, \Vvert u \Vvert^2.
\]
This directly results in the estimate
\[
  E^{N\ell^d+1} 
  = \adjustlimits \max_{\dim\V^{(N\ell^d)}=N\ell^d\ } \min_{v \in [\V^{(N\ell^d)}]^\text{c}}\, E(v)
  \ \ge\ \min_{v \in \W^\text{c}}\, E(v)
  = \min_{v \in \W^\text{c}} \frac{\Vvert v \Vvert^2}{\Vert v \Vert^2}
  \ \gtrsim\ \frac{(\ell-1)^{2}}{\eps^{2}}.
  \qedhere
\]
\end{proof}	
\subsubsection{Spectral gaps}
The combination of the above estimates shows that there will be a spectral gap of order $\calO(1)$ after the first $\calO(N)$ eigenvalues with $N = (2\eps)^{-d}$. We formulate this result in form of a corollary. 
\begin{corollary}
Assume $\beta \gtrsim (\ell-1)^2 \eps^{-2}$ for a natural number $\ell>1$ and Assumption~\ref{ass_epsAlpha}. Then, the periodic setting leads to a spectral gap of the form 
\[
  \frac{E^{N k}}{E^{N \ell^d+1}}
  \lesssim \frac{k^2}{(\ell-1)^2}
\]	
for natural numbers $\ell>k\ge 1$. In particular, we have $E^1 / E^{N \ell^d+1} \lesssim (\ell-1)^{-2}$.
\end{corollary}	
This result shows that the absence of disorder leads to large eigenvalue blocks and thus, no locality of eigenfunctions can be shown. 
%
%
\subsection{Tensor product potential}\label{sect:gaps:random}
We consider a potential $V$, which generalizes the periodic setup from the previous subsection and is a special case of the general random potential studied in Sections~\ref{sec:evp}-\ref{sect_localization}. This includes different valley formations, namely cuboids of varying side lengths, cf.~Figure~\ref{fig_potentials} (upper right). We follow the ideas of the periodic setting but need to adjust the construction of the finite element space $\W$ in order to prove a spectral gap. 
\subsubsection{Description of the potential}\label{sect:gaps:random:pot}
We consider a potential $V$, composed of one-dimensional potentials. For this, define $V_1, \dots, V_d \in L^\infty(0,1)$, each based on an $\eps$-partition of $(0,1)$ with values $0$ and $1$ only. Then, $V \in L^\infty(D)$ is given by 
\[
  V(x) := \beta + (\alpha-\beta) \big[V_1(x_1)\cdot\dots\cdot V_d(x_d) \big].
\]
Note that this construction provides non-overlapping $\alpha$-valleys in form of cuboids, each surrounded by at least one $\eps$-layer of $\beta$ values. We characterize the cuboids by their {\em shortest} side length and $N_j$ denotes the number of such valleys with minimal side length~$\eps j$. We define $D_{\alpha,j}$ as the union of these $N_j$ valleys with width~$\eps j$. The maximal (shortest) side length of an $\alpha$-valley is $\eps L$ and thus, $N_L\ge 1$. Furthermore, we need a bound for the possible anisotropy of $\alpha$-valleys of a certain size. Let $\rho_{\tilde \ell}$ denote the maximal quotient of maximal and minimal side length of all $\alpha$-valleys with width greater or equal to $\tilde{\ell}$. 
\begin{remark}
With the presented construction of $V$ we obtain the periodic potential of Section~\ref{sect:gaps:periodic} by the choice $V_j = [\, 1,\, 0,\, 1,\,0,\, \dots,1,\, 0\, ]$ for all $j=1,\dots, d$.
\end{remark}

\begin{remark}
\label{rem_expectation_Nell}
In the range $\ell \approx \log(1/\eps)$ the expectation of $N_\ell$ is of order $(2^{-\ell}/\eps)^d$ and thus, of order $\calO(1)$. In this setting, we also have $\rho_{\tilde \ell} = \calO(1)$ with high probability.
\end{remark}
Although the given setting is more obscure than the periodic case, we will show that there is -- with high probability -- a spectral gap already after $\calO(\eps^{-p})$ with $p<d$, instead of $\calO(\eps^{-d})$ eigenvalues. To prove this, we need to exploit the disorder of the potential. 
The overall strategy is to prove upper bounds for the first $N_L$ eigenvalues and lower bounds for higher energy levels. For this, we apply once more the max-min principle and construct a specific finite-dimensional subspace based on the valley formations of the potential. 
\subsubsection{Quasi-interpolation operator}\label{sect:gaps:random:quasiInt}
To obtain lower bounds on the eigenvalues we construct a finite element subspace and a corresponding quasi-interpolation operator similar to Section~\ref{sect:gaps:periodic}. For this, we fix a level $1<\tilde\ell<L$ and consider partitions of all sets $D_{\alpha,j}$ with $j>\tilde{\ell}$, using a mesh size $h \approx \eps\tilde\ell$, cf.~the illustration in  Figure~\ref{fig_lowerBound_hat}. 

More precisely, on an $\alpha$-valley $Q$ with dimensions $\eps j_1\le\dots\le \eps j_d$ and $\tilde\ell<j_1$ we consider a Cartesian mesh defined by $\Pi_{k=1}^d \bigl(\lfloor j_k / \tilde\ell \rfloor +2\bigr)$ equally distributed nodes. We extend this mesh by one more layer of elements of width $\eps/2$ into the surrounding $\beta$ region, cf.~Figure~\ref{fig_lowerBound_hat}.  A local projection operator $\Pi_{\tilde{\ell}}$ onto $Q_1$-basis functions is defined by prescribing values at nodes in $\overline{Q}$ for each $\alpha$-valley with minimal side length $j>\tilde{\ell}$ via 
$$(\Pi_{\tilde \ell} u)\vert_Q=\ISZ (u\vert_Q)$$ and enforcing vanishing traces at the boundary of each local mesh to ensure $\V$-conformity.

\begin{figure}
\begin{center}
\begin{tikzpicture}[scale=0.2]
%
\draw[red] (14.5, 0) -- (16, 5) -- (20, 0);
\draw (16, 0) -- (20, 5) -- (24, 0);
\draw[red] (20, 0) -- (24, 5) -- (25.5, 0);
\draw[red] (34.5, 0) -- (36, 5) -- (40, 0);
\draw (36, 0) -- (40, 5) -- (44, 0);
\draw (40, 0) -- (44, 5) -- (48, 0);
\draw (44, 0) -- (48, 5) -- (52, 0);
\draw (48, 0) -- (52, 5) -- (56, 0);
\draw[red] (52, 0) -- (56, 5) -- (57.5, 0);

%
\node at (8.5, -1.5) {$\eps\tilde\ell$};
\node at (20, -1.5) {$\eps(\tilde\ell+1)$};
\node at (46, -1.5) {$\eps L$};
\node at (38, -1.8) {\small $h$};
\draw[decorate,decoration={brace}] (40, -0.35) -- (36, -0.35);
\node[red] at (35.25, -1.9) {\small $\tfrac \eps 2$};
\draw[red,decorate,decoration={brace}] (36, -0.35) -- (34.5, -0.35);
%
\draw[very thick] (6, 0) -- (11, 0);
\draw[thin] (14, 0) -- (26, 0);
\draw[very thick] (16, 0) -- (24, 0);
\draw[thin] (34, 0) -- (58, 0);
\draw[very thick] (36, 0) -- (56, 0);
\foreach \x in {6, 11, 16, 24, 36, 56} {
	\draw[very thick] (\x, -0.4) -- (\x, 0.4);
}
\node at (30, 0) {\Large $\dots$};
\end{tikzpicture} 
\end{center}
\caption{Basis functions of $\W_{\tilde{\ell}}$ on large $\alpha$-valleys with mesh size $h \approx \eps \tilde{\ell}$ in the one-dimensional setting. The basis functions at the boundary are slightly deformed. }
\label{fig_lowerBound_hat}
\end{figure}
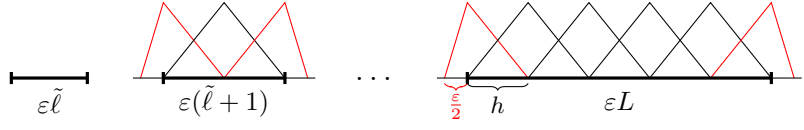
The image of $\Pi_{\tilde \ell}$ defines the finite element space $\W_{\tilde{\ell}}$ with a dimension bounded by 
\begin{align}
\label{def_K}
  \dim \W_{\tilde{\ell}}
  \, =:\, K 
  \, \lesssim\, \rho_{\tilde \ell}^{d-1}\sum\nolimits_{j=\tilde \ell+1}^L N_j\, \big\lfloor j /{\tilde\ell} \big\rfloor^{d}.
\end{align}
We emphasize that $\rho_{\tilde \ell}$ enters the estimate as we characterized the $\alpha$-valleys by means of their shortest side length. 

As in the periodic case, the kernel of the projection operator $\Pi_{\tilde{\ell}}$ defines  an appropriate complement space $\W_{\tilde{\ell}}^\text{c}$ to be used in connection with eigenvalue estimates via the max-min characterization. 
Finally, we derive for $u\in \V$ an approximation estimate of the form
\begin{equation}
\label{eq_qi}
  \Vert u - \Pi_{\tilde{\ell}} u\Vert 
  \, \lesssim\, \eps\tilde \ell\, \Vvert u \Vvert. 
\end{equation}
To see this, we split the left-hand side into
\begin{align*}
  \Vert u - \Pi_{\tilde{\ell}} u \Vert^2
  \le \sum\nolimits_{j\le \tilde\ell}\, \Vert u \Vert_{D_{\alpha,j}}^2 
     + \sum\nolimits_{j> \tilde\ell}\, \Vert u - \ISZ u \Vert_{D_{\alpha,j}}^2 
     + 2\, \Vert u \Vert^2_{\Ob} + 2\, \Vert \Pi_{\tilde{\ell}} u \Vert^2_{\Ob}.
\end{align*}
For small $\alpha$-valleys with width $j\le \tilde{\ell}$ we apply the Poincar\'e-Friedrich's inequality to~$\eta u$, where $\eta$ denotes a cut-off function similarly as in Section~\ref{sec:evp:cutoff}. For that we note that $\eta u \in H^1_0(\hspace{1pt}N(D_{\alpha,j})\hspace{1pt})$, where $N(D_{\alpha,j})$ equals $D_{\alpha,j}$ extended by a surrounding $\tfrac \eps 2$-layer. With this, the Poincar\'e-Friedrich's inequality yields
\begin{equation}
\label{eq_smallValleys}
  \Vert u \Vert_{D_{\alpha,j}} 
  \le \Vert \eta u \Vert_{N(D_{\alpha,j})}
  \lesssim  \eps\tilde\ell\  \Vert \nabla (\eta u) \Vert_{N(D_{\alpha,j})}
  \lesssim\, \eps\tilde\ell\ \Vvert u \Vvert_{N(D_{\alpha,j})}.
\end{equation}
In the last step we have used the same arguments as in the proof of estimate \eqref{estimate_L2norm}.

On $\alpha$-valleys with width $j> \tilde{\ell}$ we can directly apply the approximation property of the Scott-Zhang operator \cite{HeuS07}. 
Finally, for $T$ being an element of the $\beta$-region of width~$\tfrac\eps 2$, we show that $\Vert \Pi_{\tilde \ell}\, u \Vert_T \lesssim \eps\tilde\ell\, \Vvert u \Vvert_{N(T)}$. In this case, $N(T)\subset \Ob$ denotes the union of $T$ and all its neighbors within the $\tfrac \eps 2$-layer in $\Ob$, surrounding an $\alpha$-valley. Note that, due to the construction of the operator $\Pi_{\tilde \ell}$, we need to estimate $u$ along edges $E$. Using the trace identity of \cite{CarGR12}, we get for an edge $E$ of $T$,  
\[
  \frac {1}{|E|} \int_E u \dx
  \le \frac{1}{|T|} \int_T |u| \dx + \frac{\eps\tilde\ell}{|T|} \int_T |\nabla u| \dx
  \le \frac{1}{|T|^{1/2}} \Vert u\Vert_{T} + \frac{\eps\tilde\ell}{|T|^{1/2}} \Vert \nabla u \Vert_T.
\]
Considering appropriate edges, such estimates lead to 
\[
  \Vert \Pi_{\tilde \ell}\, u \Vert_T^2
  \lesssim |T|\, \sum_{\tilde T \in N(T)} \frac{1}{|\tilde T|} \Big( \Vert u\Vert^2_{\tilde T} + (\eps\tilde\ell)^2 \Vert \nabla u \Vert^2_{\tilde T} \Big)
  \le \sum_{\tilde T \in N(T)} (\eps\tilde\ell)^2\, \Vvert u \Vvert^2_{\tilde T}
  \le (\eps\tilde\ell)^2\, \Vvert u \Vvert^2_{N(T)}.
\]
Here we used again that $N(T) \subset \Ob$ and that $\beta^{-1} \lesssim \eps^2$.
\subsubsection{Eigenvalue bounds}\label{sect:gaps:random:bounds}
As in the periodic setting, we consider eigenfunctions of the shifted Laplace eigenvalue problem with homogeneous Dirichlet boundary conditions on the $N_L$ $\alpha$-valleys with width $\eps L$. On such a valley we know that the first eigenvalue equals $\alpha + d\pi^2/(\eps L)^2$. If we extend the corresponding eigenfunctions by zero, we obtain $N_L$ functions $\hat u_1, \dots, \hat u_{N_L} \in\V$ for which we have $E(\hat u_j)= \alpha + d\pi^2/(\eps L)^2$. With $\alpha \lesssim (\eps L)^{-2}$ from Assumption~\ref{ass_epsAlpha}, we obtain by Lemma~\ref{lem_evBound} the eigenvalue bound 
\[
  E^{N_L} \lesssim \frac{1}{(\eps L)^2}. 
\]
To ensure a spectral gap, we also need lower bounds on the eigenvalues. Using \eqref{eq_smallValleys} and the approximation property of $\ISZ$, we estimate for $u \in \W_{\tilde{\ell}}^\text{c}$, i.e., $\Pi_{\tilde\ell} u = 0$, 
\begin{align}
\label{lowerbound-E-complementary-space}
  \Vert u \Vert^2
  &= \sum\nolimits_{j\le \tilde\ell}\, \Vert u \Vert_{D_{\alpha,j}}^2 
    + \sum\nolimits_{j> \tilde\ell}\, \Vert u - \ISZ u \Vert_{D_{\alpha,j}}^2 
    + \Vert u \Vert^2_{\Ob} \\
  \nonumber&\lesssim (\eps\tilde\ell)^2\, \sum\nolimits_{j\le \tilde\ell}\, \Vvert u \Vvert^2_{N(D_{\alpha,j})} 
    + (\eps\tilde\ell)^2\, \sum\nolimits_{j> \tilde\ell}\, \Vert \nabla u \Vert_{D_{\alpha,j}}^2 + \frac{1}{\beta} \Vert u \Vert^2_{V, D_\beta} 
  \lesssim (\eps\tilde\ell)^2\, \Vvert u \Vvert^2. 
\end{align}
Note that we have used here once more Assumption~\ref{ass_epsBeta}. The application of the max-min eigenvalue characterization in \eqref{eqn_lambda_Nell} then proves 
\[
  E^{K+1}
  = \adjustlimits \max_{\dim\V^{(K)}=K\ } \min_{v \in [\V^{(K)}]^\text{c}}\, E(v)
  \ \ge\ \min_{v \in \W_{\tilde{\ell}}^\text{c}}\, E(v)
  \ \gtrsim\  \frac{1}{(\eps\tilde\ell)^{2}}.
\]	
A combination of the lower and upper bounds of the eigenvalues yields a guaranteed spectral gap of order $\calO(1)$ within the first $K+1$ eigenvalues, if $\tilde\ell$ is sufficiently small compared to $L$. 

\begin{corollary}
\label{cor_generalCase}
In the considered setting including Assumptions~\ref{ass_epsBeta} and~\ref{ass_epsAlpha} we have an estimate of the form 
\begin{equation}\label{eq_ellleqL}
  \frac{E^{N_L}}{E^{K+1}}
  \ \le\ c_1 \frac{\eps^2\tilde\ell^2}{\eps^2 L^2} 
  \ =\ c_1 \frac{\tilde\ell^2}{L^2} =: q,
\end{equation}
for some generic constant $c_1$. Thus, for sufficiently small $\tilde \ell$, i.e. $\tilde{\ell}^2 \le q \hspace{2pt}c_1^{-1} L^2 $ with $0<q<1$, we obtain a spectral gap of size $q<1$.
\end{corollary}
\begin{remark}
\label{rem_Kissmall}
Corollary~\ref{cor_generalCase} shows that any choice $\tilde{\ell}^2 < L^2 / c_1$ ensures a reasonable spectral gap~$q<1$. 
Here, $c_1>0$ is the multiplicative constant in \eqref{eq_ellleqL}. On the other hand, we also need to ensure that $K=\dim\W_{\tilde \ell}$ is sufficiently small.
Since the probability of an $\alpha$-valley to be of size $\eps L$ is of the order $(2^{-L}/\eps)^d$, $L$ is with high probability larger than $\log(c_2/\eps)$ with some other constant $c_2>0$ that has an expectation of $c_2=1/4$. 
Next, we choose $\tilde{\ell}= \lceil L/\sqrt{q^{-1} c_1}\rceil$, leading to $K = \calO(\eps^{-p})$ for $p<d$ (with high probability). 
To detail this claim, recall that $N_{\tilde{\ell}} \lesssim (2^{-\tilde \ell}/\eps)^d$ and $\rho_{\tilde \ell} \lesssim 1$ with high probability, cf.~Remark~\ref{rem_expectation_Nell}. Hence, with $c_q:= \sqrt{q^{-1} c_1}$ we have 
\[
  K 
  \lesssim \sum\nolimits_{j=\tilde \ell+1}^L N_j 
  \lesssim \eps^{-d} 2^{-\tilde \ell d}
  \le \eps^{-d} 2^{-Ld / c_q }  
  = c_2^{-d / c_q } \eps^{-d} \eps^{d / c_q }
  = c_2^{-d / c_q } \eps^{-d \frac{c_q-1}{c_q}} = \mathcal{O}(\eps^{-p}), 
\]
where $p := d \frac{c_q-1}{c_q} < d$.
Note that in this setting, $D_{\alpha,j}$ is the union of $\calO(\eps^{-p})$ cuboids of side length $\eps j$ for $j>\tilde{\ell}$. As a result, $\W_{\tilde \ell}$ is a local space that covers a domain that has a volume of order $\calO(\eps^{d-p})$ up to logarithmic terms, where $d-p>0$. Hence, the support of functions in $\W_{\tilde \ell}$ is asymptotically vanishing for $\eps \rightarrow 0$.
\end{remark}
\begin{remark}
Numerical examples indicate that -- in the presence of disorder -- $K$ is even $\calO(1)$. 
In the one-dimensional example presented in Figure~\ref{fig:spectra}, he have $\gap\le\frac 12$ with $K=2$.
\end{remark}
For the construction of a suitable set of starting functions  $V^{(0)}$ in the next subsection, we also need an upper bound on $E^K$. In an $\alpha$-valley $Q \subset D_{\alpha, j}$, $j>\tilde{\ell}$, we consider the first $r$ eigenvalues of the shifted Laplacian on $Q$, where $r$ equals the number of degrees of freedom in $\W_{\tilde{\ell}}$, associated with the $\alpha$-valley $Q$. Note that $r$ depends on the anisotropy of the given valley characterized by $\rho_{\tilde \ell}$. Then, $Q$ having the dimensions $\eps j_1 \le \dots \le \eps j_d$ with $j=j_1>\tilde \ell$ implies $j_d/j_1 \le \rho_{\tilde \ell}$ and 
\[
  r 
  \, \approx\,  \frac{j_1 \cdot \dots \cdot j_d}{\tilde \ell^d}
  \, \le\, \rho_{\tilde \ell}^{d-1} \frac{j^d}{\tilde \ell^d}.
\]
Extending these $r$ eigenfunctions by zero, we obtain (considering all $\alpha$-valleys) in total~$K$ orthogonal functions with an energy bounded by 
\[
  \alpha + \frac{\pi^2}{(\eps j)^2} r^{2/d}
  \, \lesssim\, \frac{1}{(\eps j)^2} \frac{j^2}{\tilde \ell^2} \rho_{\tilde \ell}^{2(d-1)/d}
  \, \le\, \rho_{\tilde \ell}^{4/3} \frac{1}{(\eps \tilde \ell)^2}  
\]
and thus, by Lemma~\ref{lem_evBound}, 
\begin{equation}
\label{rem:upperbound_K}
  E^K \lesssim \rho_{\tilde \ell}^{4/3} (\eps \tilde{\ell})^{-2}.
\end{equation}
\subsubsection{Starting subspace}\label{sect:gaps:random:starting}
To apply the convergence result of Theorem~\ref{theorem-convergence-inexact-block} in the present setting, it remains to construct a suitable set of starting functions $V^{(0)}$ with an invertible matrix $\calC$, cf.~Section~\ref{sect_localization_block}. 
Here the idea is to extend the space $\W_{\tilde \ell}$ from Section~\ref{sect:gaps:random:quasiInt} by a few levels of $\alpha$-valleys such that it remains local and then project the first eigenfunctions into this space. 

We introduce the finite element space $\W_{\tilde m}$, which is defined as $\W_{\tilde \ell}$ but considers all $\alpha$-cuboids with (minimal) side length $\eps j$, $j> \tilde m$, for some parameter $\tilde m < \tilde \ell$. 
The local mesh size then equals~ $h\approx \eps\tilde m$. 
Thus, the extension enlarges the number of $\alpha$-valleys but also refines the previous mesh. The corresponding quasi-interpolation operator reads $\Pi_{\tilde m}\colon \V \to \W_{\tilde m}$ and the dimension of $\W_{\tilde m}$ is bounded by 
\[
  \dim \W_{\tilde \ell}
  \, \le\, \dim \W_{\tilde m}
  \, =:\, K_{\tilde m} 
  \, \lesssim\, \rho_{\tilde m}^{d-1}\sum\nolimits_{j=\tilde m+1}^L N_j\, \big\lfloor j /{\tilde m} \big\rfloor^{d}.
\]
Recall that $N_j$ denotes the number of $\alpha$-cuboids with minimal side length $\eps j$.

We show that this construction leads to an invertible matrix~$\calC$. 
\begin{lemma}
\label{lem:startingV}
Consider a tensor product potential as described in Section~\ref{sect:gaps:random:pot} together with Assumptions~\ref{ass_epsBeta} and~\ref{ass_epsAlpha}. 	
Further, we define the starting space $V^{(0)}$ as the $L^2$-projection $P_{\tilde{m}}$ of the first $K$ eigenfunctions into $\W_{\tilde m}$, i.e., we set
$$
v_j^{(0)} := P_{\tilde{m}} u_j,
$$
where $u_j$ denote the first (normalized) eigenfunctions, $j=1,\dots, K$. This choice leads to an invertible matrix $\calC$ with~$\| \calC^{-1} \|_1 \lesssim \eps^{-r}$ for some power $r>0$. Hence, it is uncritical for the localization.
\end{lemma}

\begin{proof}	
The entries of the matrix $\calC$ are given by
$$
\alpha_{i,j} = (u_i , v_j^{(0)}  ) = (P_{\tilde{m}} u_i , P_{\tilde{m}} u_j ) =  ( v_i^{(0)}  , v_j^{(0)}  ).
$$
Hence, $\calC$ is a symmetric mass matrix and invertible if the functions $v_j^{(0)}$ are linearly independent. The linear independence follows from the injectivity of $P_{\tilde{m}}$ on the span of the first $K$ eigenfunctions, which is proved by contradiction. Assume that there exists $\mathbf{a} \in \R^K \setminus \{ 0\}$ and $u := \sum_{j=1}^K \mathbf{a}_j u_j$ such that $P_{\tilde{m}}u=0$. 
Then $u$ is in the $L^2$-orthogonal complement of $\W_{\tilde m}$ and we have $\| u \| = \| u - P_{\tilde{m}} u \| \le \| u - \Pi_{\tilde{m}} u \| $. This and \eqref{eq_qi} show that 
\begin{equation*}
  E(u) \gtrsim (\eps \tilde{m})^{-2}.
\end{equation*}
From Lemma~\ref{lem_RayleighBound} and \eqref{rem:upperbound_K} we also have the upper bound 
\begin{equation*}
  E(u) \lesssim (\eps \tilde{\ell})^{-2}.
\end{equation*}
If $\tilde m$ is sufficiently small compared to $\tilde \ell$, the previous lower and the upper energy bounds are contradictive. Hence, $P_{\tilde{m}}$ is injective and the functions $v_j^{(0)}$ are linearly independent, which in turn implies the regularity of the matrix $\calC$. 
Since $\| \calC \|_2^{-1} = \| \calC^{-1} \|_2$, we can bound the $1$-norm by   
\begin{align*}
\| \calC^{-1} \|_1 \le c_K \| \calC \|_1^{-1} \le c_K \Big( \max_{1\le j \le K} |\alpha_{j,j}|\ \Big)^{-1}
\end{align*}
with some constant $c_K$ that depends at most polynomially on $K$ through norm equivalence. Using \eqref{eq_qi} and $\Vert u_j -  P_{\tilde{m}} u_j \Vert   \le \Vert u_j -  \Pi_{\tilde{m}} u_j \Vert   $ we have
\begin{align*}
  \alpha_{j,j}
  = 1 - (u_j, u_j- P_{\tilde{m}} u_j)
  \ge 1 - \Vert u_j -  \Pi_{\tilde{m}} u_j \Vert 
  \ge 1 - c\, \eps \tilde{m}\, \Vvert u_j \Vvert 
  = 1 - c\, \eps\tilde{m}\, \sqrt{E^j}.
\end{align*}
The upper bound $E^K \lesssim \rho_{\tilde \ell}^{4/3} (\eps \tilde{\ell})^{-2}$ from \eqref{rem:upperbound_K} with anisotropy constant $\rho_{\tilde \ell}$ then implies 
\[
  \alpha_{j,j}
  \ge 1 - c\, \eps\tilde{m}\, \sqrt{E^K}
  \ge 1 - c' \tilde{m}/\tilde \ell.
\]
Hence, we can bound $\| \calC^{-1} \|_1$ by a constant that depends at most polynomially on $K$, respectively polynomially on $\eps^{-1}$.

Recall once more that the space $\W_{\tilde m}$ is with high probability local, cf.~Remark~\ref{rem_Kissmall} where the argument is elaborated. 
This implies that $V^{(0)}$ is a suitable starting subspace. 
\end{proof}
With this, all assumptions of Theorem \ref{theorem-convergence-inexact-block} (and Remark \ref{remark-on-theorem-convergence-inexact-block}) are verified. Starting from an initial space that is spanned the few localized basis functions in $V^{(0)}$, we can approximate the first $N_L$ eigenfunctions of $\calH$ in $\calO(\log(1/\tol) /\log(1/\gap))$ steps with an accuracy of order $\tol$ times the energy of the functions in $V^{(0)}$. 
Since the support of the starting space only increases slowly during the iteration process, we conclude that the eigenfunctions are well approximated in a domain that has a volume of order $\calO(\eps^{d-p})$, up to logarithmic terms, with $d-p>0$. 
Hence, the exponential localization of the eigenfunctions is shown asymptotically for vanishing  $\eps \rightarrow 0$. 
\begin{remark}
In the range $\beta \approx \eps^{-2}$ we have the stability estimate $\Vvert \Pi_{\tilde{m}} u\Vvert \lesssim \tilde m^2 \Vvert u\Vvert$ and thus, the functions in $V^{(0)}$ satisfy $\Vvert v_j^{(0)}\Vvert = \Vvert \Pi_{\tilde{m}} u_j\Vvert \lesssim \log^2(1/\eps) \Vvert u_j \Vvert$. 
\end{remark}
%
%
%
\subsection{Domino block potential}\label{sect:gaps:domino}
As a third example we consider a potential that is formed by a disordered domino block structure. This example aims to demonstrate how our technique can be applied if the $\alpha$-valleys are not surrounded by $\beta$-layers, i.e., a setting that cannot be reduced to a quasi-one-dimensional case. In order to prove the existence of relevant spectral gaps we again follow the ideas from the previous two examples. 
\subsubsection{Description of the potential}
To make the setting precise we call $B$ a $j$-domino block (or simply $j$-block) if it is a closed cuboid in $\R^d$ consisting of elements in $\calT$, which is composed of an $\alpha$-cube with side length $\eps j$ and a $\beta$-cube with the same side length. Hence, such a cuboid has the dimension $2\eps j$ in one space direction and $\eps j$ in all other directions. Such a domino block has no preferred orientation. We shall now assume that the potential is formed by a non-overlapping union of such $j$-domino blocks where $j\in\N$ can take any value between $1$ and $L$. An example for such a potential is given in Figure~\ref{fig_potentials} (lower left). The set of all these blocks, which then defines the potential $V$, is denoted by $\mathcal{B}$ and the set of all $j$-domino blocks by $\mathcal{B}_j$. We observe that
$$
  \overline{D} 
  \, =\, \bigcup\nolimits_{B\in \mathcal{B}} B 
  \, =\, \bigcup\nolimits_{j=1}^L \bigcup\nolimits_{B\in \mathcal{B}_j} B.
$$
We further assume that the probability of finding a small domino block is much higher then finding a large domino block. More precisely, we assume that when selecting a domino block from $\mathcal{B}$ the expectation that it is a $j$-block is approximately $2^{-j}$. Note that the parameter $L$ is, except for unlikely situations, the same as in Section~\ref{sec:evp:cutoff}. For a $j$-domino block $B \in \mathcal{B}_j$ we denote the cube where $V$ is equal to $\alpha$ by $B_{\alpha}\subset B$. 
\begin{remark}
Since the total number of domino blocks in $D$ can be at most of order $\eps^{-d}$, we conclude that the expected number of $j$-blocks is of order  $2^{-j}\eps^{-d}$. Hence, for any $j\ge d\, \log(1/\eps)$ we expect the number of $j$-blocks to satisfy 
$$
N_j := |\mathcal{B}_j| = \calO(1).
$$
\end{remark}
As for the tensorized potential in Section~\ref{sect:gaps:random:bounds}, we will show the existence of a relevant spectral gap after the first $K$ eigenvalues. 
For this, we again prove upper and lower eigenvalue bounds and construct certain finite element spaces in the largest valleys. 
\subsubsection{Quasi-interpolation operator}\label{sect:gaps:domino:quasiInt}
Once more we need to introduce a suitable local finite element space on a subset of $D$ and to exploit the approximation properties of the Scott-Zhang quasi-interpolation operator. 
For that purpose, let us fix a level $1<\tilde\ell<L$ satisfying $\tilde\ell\gtrsim \log(1/\eps)$. With this, we consider all sets $\mathcal{B}_j$ of $j$-blocks at level $j>\tilde{\ell}$. Next, we restrict our attention to the $\alpha$-parts of these domino blocks and define the set
$$
  D_{\alpha,\, >{\tilde\ell}} \, :=\, 
  \bigcup\nolimits_{j=\tilde \ell + 1}^{L} D_{\alpha, j},  
  \qquad \mbox{where }\ D_{\alpha, j}:= \bigcup\nolimits_{B \in \mathcal{B}_j } B_{\alpha}. 
$$ 
On $D_{\alpha,\, >{\tilde\ell}}$ we introduce a uniform Cartesian mesh with mesh size $h\approx \eps\tilde\ell$ and extend it by one more layer of elements with width $\eps/2$, cf.~Section~\ref{sect:gaps:random:quasiInt}. Observe that the extended mesh will intersect the $\beta$-parts of all contributing domino blocks, but it will also intersect other domino-blocks on possibly lower levels. 
The space occupied by the extended mesh is denoted by $\tilde{D}_{\alpha,\,  >{\tilde\ell}}$. On $\tilde{D}_{\alpha,\, >{\tilde\ell}}$ we consider the arising $Q_1$-finite element space based on the extended mesh and with homogeneous Dirichlet boundary conditions, as before. The resulting space is denoted by $\W_{\tilde{\ell}}$ and we have
\begin{align*}
  \dim \W_{\tilde{\ell}}
  \, =:\, K 
  \, \lesssim\, \sum\nolimits_{j=\tilde \ell+1}^L N_j\, \big\lfloor j /{\tilde\ell} \big\rfloor^{d},
\end{align*}
which is expected to be of order $\calO(\eps^{-p})$ for some $p<d$, cf.~Remark~\ref{rem_Kissmall}. 
Note that this also implies that $\W_{\tilde{\ell}}$ is a local space, since its basis functions are only supported on few subsets with maximum diameter~$\eps L$. 
Based on the Scott-Zhang operator $\ISZ$, restricted to the mesh on $D_{\alpha,\, >{\tilde\ell}}$, we uniquely define the local projection operator 
$$
\Pi_{\tilde{\ell}} : \V \rightarrow \W_{\tilde{\ell}}
$$
by means of the property $\Pi_{\tilde \ell} u = \ISZ u$ on $D_{\alpha,\,  >{\tilde\ell}}$. As in the previous example, $\Pi_{\tilde{\ell}}$ prevents that information spreads from $\beta$-regions into $D_{\alpha,\, >{\tilde\ell}}$. At the same time it also prevents that information from smaller domino blocks spreads into $D_{\alpha,\, >{\tilde\ell}}$. The key to the analysis is again an interpolation error estimate.
\begin{lemma}\label{lemma:interpol-est:domino}
Consider the domino block potential from above under Assumption \ref{ass_epsBeta}. Then, for the local projection operator $\Pi_{\tilde{\ell}}$ and all $u \in \V$ it holds that
\begin{align*}
  \Vert u - \Pi_{\tilde \ell}\, u \Vert 
  \, \lesssim\,  \eps\tilde\ell\, \Vvert u \Vvert.
\end{align*}
\end{lemma}
\begin{proof}
We split the domain into three parts 
$$
  D = D_{\alpha,\, >{\tilde\ell}} \hspace{3pt} 
    \cup \hspace{3pt} \big(D \setminus  \tilde{D}_{\alpha,\, >{\tilde\ell}} \big) \hspace{3pt} \cup \hspace{3pt} \big(\tilde{D}_{\alpha,\, >{\tilde\ell}} \setminus D_{\alpha,\, >{\tilde\ell}} \big)
$$
and estimate $u - \Pi_{\tilde \ell} u$ on these sub-domains individually. On $D_{\alpha,\, >{\tilde\ell}}$, the estimate follows immediately with the properties of the Scott-Zhang operator, see Section~\ref{sect:gaps:periodic}. 
On $D \setminus \tilde{D}_{\alpha,\, >{\tilde\ell}}$ we have that $\Pi_{\tilde \ell} u=0$. Consequently, the claimed estimate reduces to an estimate of $u$. For this, we derive once more a Friedrichs-type inequality as in Lemma~\ref{lem_eta_Friedrich}. Considering averaged Taylor polynomials as in Appendix~\ref{appendix-a} and introducing a modified cut-off function that exploits the particular structure of the potential (a large $\alpha$-block is always adjacent to a $\beta$-block of the same size), we obtain
\begin{align*}
  \| u\|_{D \setminus  D_{\alpha,\, >{\tilde\ell}}} \lesssim \eps \ell \hspace{2pt} \Vvert u\Vvert_{D \setminus  D_{\alpha,\, >{\tilde\ell}}}.
\end{align*}
We emphasize that this estimate is free from pollution constants $\kappa_\calT$ and ${\tilde \ell}^d$ that occur in the general case.
Finally, for the estimate in $\tilde{D}_{\alpha,\, >{\tilde\ell}} \setminus D_{\alpha,\, >{\tilde\ell}}$ we can proceed analogously as in Section~\ref{sect:gaps:random:quasiInt}. In particular, for any element $T$ of the extended mesh that lies in the layer $\tilde{D}_{\alpha,\, >{\tilde\ell}} \setminus D_{\alpha,\, >{\tilde\ell}}$ we have 
\begin{align*}
\Vert \Pi_{\tilde \ell}\, u \Vert_T^2
\lesssim |T|\, \sum_{\tilde T \in N(T)} \frac{1}{|\tilde T|} \Big( \Vert u\Vert^2_{\tilde T} + (\eps\tilde\ell)^2 \Vert \nabla u \Vert^2_{\tilde T} \Big)
\lesssim \Vert u \Vert^2_{N(T)} + (\eps\tilde\ell)^2\, \Vert \nabla u \Vert^2_{N(T)},
\end{align*}
where $N(T) \subset \tilde{D}_{\alpha,\, >{\tilde\ell}} \setminus D_{\alpha,\, >{\tilde\ell}}$ is the union of $T$ and all its neighbors in $\tilde{D}_{\alpha,\, >{\tilde\ell}} \setminus D_{\alpha,\, >{\tilde\ell}}$. The combination of all these estimates finishes the proof.
\end{proof}
With the quasi-interpolation operator in hand we are ready to establish eigenvalue bounds based on the max-min eigenvalue characterization. 
\subsubsection{Eigenvalue bounds}\label{sect:gaps:domino:bounds}
As for the tensor potential we directly obtain by Assumption~\ref{ass_epsAlpha} the upper eigenvalue bound $E^{N_L} \lesssim (\eps L)^{-2}$. By Lemma \ref{lemma:interpol-est:domino} and the max-min principle we also have
\[
  E^{K+1}
  = \adjustlimits \max_{\dim\V^{(K)}=K\ } \min_{v \in [\V^{(K)}]^\text{c}}\, E(v)
  \ \ge\ \min_{v \in \W_{\tilde{\ell}}^\text{c}}\, E(v)
  \ \gtrsim\  \frac{1}{(\eps\tilde\ell)^{2}}.
\]
Consequently, Corollary \ref{cor_generalCase} remains valid and for sufficiently small $\tilde \ell$ we obtain a spectral gap of order $\calO(1)$ as we have
\begin{equation*}
  \frac{E^{N_L}}{E^{K+1}}
  \ \lesssim\ \frac{\tilde\ell^2}{L^2}.
\end{equation*}
For further discussions on this estimate, we refer to Remark \ref{rem_Kissmall}. Furthermore, note that the arguments from Section \ref{sect:gaps:random:bounds} also allow us to derive the upper eigenvalue bound
\begin{equation*}
  E^K \lesssim (\eps \tilde{\ell})^{-2}.
\end{equation*}
\subsubsection{Starting subspace}\label{sect:gaps:domino:starting}
Finally, it remains to show the existence of a suitable set of starting functions $V^{(0)}$ to apply Theorem~\ref{theorem-convergence-inexact-block}. This can be done analogously to the case of a tensorized potential as elaborated in Section~\ref{sect:gaps:random:starting}. This can be done without modifying the arguments, thanks to the availability of the interpolation error estimate from Lemma~\ref{lemma:interpol-est:domino}. In this case, we define $V^{(0)}$ as the projection of the first $K$ eigenfunctions of the Schr\"odinger operator $\calH$ into a refined version of the local finite element space $\W_{\tilde \ell}$. 
As before, this choice also ensures the invertibility of the matrix~$\calC$, cf.~Lemma~\ref{lem:startingV}. 

In conclusion, Theorem \ref{theorem-convergence-inexact-block} and Remark \ref{remark-on-theorem-convergence-inexact-block} imply that the first $N_L$ eigenfunctions can be expressed as the union of $K$ functions that show an exponential decay in units of $\eps$. Since $K$ is (with high probability) sufficiently small compared to $\eps^{-d}$, we obtain that the first $N_L$ eigenfunctions are exponentially localized, where the localization centers are the $j$-domino blocks with the highest $j$-levels.
\section{Conclusion}
This paper provided quantitative estimates for the lowermost part of the spectrum of random Schr\"odinger operators in the PDE setting that cannot be extracting from existing theoretical studies. These findings provide a theoretical basis for the recent trend on computational studies of Anderson localization in the linear Schr\"odinger eigenvalue problem, e.g., \cite{ArnDJMF16,Steinerberger2017,LuS18,AltPV18,ArnDFJM19,2018arXiv180600565X}. Moreover, the constructive proofs inspire novel localized computational approaches for the approximation of localized states.  
The block inverse iteration used in the proof can be turned into a fast algorithm for the computation of eigenstates and spectra as presented, e.g., in~\cite{AltP19}.

In addition to the particular results on the Schr\"odinger eigenvalue problem, the paper illuminates the large potential of classical tools of numerical analysis such as domain decomposition and the theory of iterative solvers for the mathematical analysis of multiscale partial differential equations and the corresponding eigenvalue problems. Since Anderson localization is an almost universal wave phenomenon known also for sound \cite{2008NatPh...4..945H} and  electromagnetic waves (in particular light) \cite{1997Natur.390..671W} we believe that the techniques presented here will be useful in many other physical contexts. 
%
%
\newcommand{\etalchar}[1]{$^{#1}$}

%
%
\begin{appendix}
\section{A Friedrichs-type estimate for truncated functions}
\label{appendix-a}
In the following we prove the Friedrichs inequality of the form
$$
\| \eta v \|_{L^2(D)} \lesssim \kappa_\calT L^d \hspace{2pt} \eps L \hspace{2pt} \|\nabla (\eta v) \|_{L^2(D)},
$$
for truncated functions $\eta v$, as formulated in Lemma \ref{lem_eta_Friedrich}.
%
\begin{proof}[Proof of Lemma \ref{lem_eta_Friedrich}]
Let $q \in \calTb$ be a $\beta$-cube with edge length $\eps$ and barycenter $x_q$. Denoting $\rho:= \eps/4,$ we know that the ball $B_{\rho}(x_{q})$ with radius $\rho$ and center $x_q$ is fully contained in $q$. This is illustrated in Figure \ref{fig_appendix_label}. 

In this proof we shall exploit the theory of so-called averaged Taylor polynomials, cf.~\cite[Ch.~4.1]{BreS08}. For a function $w\in \V=\Hper$ we let $T(w)$ denote the first-order averaged Taylor polynomial obtained by weighted averages over the ball $B_{\rho}(x_{q}) $. More precisely, we define $T(w)$ by
$$
T(w) = \int_{ B_{\rho}(x_{q}) } w(y) \phi(y) \dy \equiv \mbox{const}
$$
for some cut-off function $\phi$. Though the precise shape of $\phi$ is not relevant for our purposes, we assume that it is smooth, $\phi \in C^{\infty}_0(\R^d)$, that it only has support on the ball $B_{\rho}(x_{q})$ and that it has the average $\int_{\R^d} \phi(y) \hspace{2pt} dy =1$. Since we have by construction that $w=\eta v \equiv 0$ in $B_{\rho}(x_{q})$, we see that $\eta v$ does not intersect the support of $\phi$. Consequently we trivially have $T(\eta v) \equiv 0$. Using this observation, we obtain the following representation of $(\eta v)(x)$ for any $x \in D$ (as proved in \cite[Prop.~4.2.8]{BreS08}):
\begin{align}
\label{representation-eta-v}
(\eta v)(x) = (\eta v)(x) - T(\eta v)(x)  =  \int_{C_{q}(x)}  k(x,z) \hspace{2pt} (x-z) \cdot \nabla (\eta v)(z)  \dz.
\end{align}
Here, $C_{q}(x)$ is the convex hull of the point $\{x\}$ and the ball $B_{\rho}(x_{q})$ (cf. Figure \ref{fig_appendix_label}) and $k$ is a function fulfilling
\begin{align}
\label{est-k-brenner-scott}
|k(x,z)| \lesssim \left( 1 + \frac{|x-x_{q}|^d }{\rho} \right) |z-x|^{-d}.
\end{align}
Let  $\hat{D}_q \subset \R^d$ be an arbitrary subdomain that is star-shaped with respect to the ball $B_{\rho}(x_{q})$. Then for any $x \in \hat{D}_q$ we have that the cone $C_{q}(x)$ is fully contained in $\hat{D}_q$ (see again Figure \ref{lem_eta_Friedrich} for illustration) and hence we obtain
\begin{eqnarray*}
\| \eta v \|_{L^2(\hat{D}_q)}^2
&{\overset{\eqref{representation-eta-v}}{=}}&
{\int_{\hat{D}_q} \left( \int_{C_{q}(x)}  k(x,z) \hspace{2pt} (x-z) \cdot \nabla (\eta v)(z)  \dz \right)^2 \hspace{2pt} dx}\\
&\overset{{\eqref{est-k-brenner-scott}}}{\lesssim}& \int_{\hat{D}_q} \left( \int_{C_{q}(x)}   \left( 1 + \frac{|x-x_{q}| }{\rho} \right)^d |z-x|^{-d+1} |\nabla (\eta v)(z)| \dz \right)^2 \dx \\
&\lesssim& \Big( 1 + \frac{\mbox{dist}(x_{q} , \partial \hat{D}_q)^d }{\rho} \Big)^2 \int_{\hat{D}_q} \left( \int_{C_q(x)}    |z-x|^{-d+1} |\nabla (\eta v)(z)| \dz \right)^2 \dx \\
&{\overset{C_q(x) \subset \hat{D}_q}{\lesssim}}& \Big( 1 + \frac{\mbox{dist}(x_{q} , \partial \hat{D}_q) }{\rho} \Big)^{2d} \sup_{ x \in \hat{D}_q } \left( \int_{\hat{D}_q} |z-x|^{-d+1} \dz \right)\\
&\enspace&\hspace{45pt} \cdot \int_{\hat{D}_q} \int_{\hat{D}_q}   |z-x|^{-d+1} |\nabla (\eta v)(z)|^2 \dz \dx \\
&\lesssim& \Big( 1 + \frac{\mbox{dist}(x_{q} , \partial \hat{D}_q) }{\rho} \Big)^{2d} \sup_{ x \in \hat{D}_q } \left( \int_{\hat{D}_q} |z-x|^{-d+1} \dz \right)^2 \|\nabla (\eta v) \|_{L^2(\hat{D}_q)}^2. 
\end{eqnarray*}

\begin{figure}
\begin{center}
\begin{tikzpicture}[scale=1.5]
	\foreach \point in {(0,4), 
					    (1,5), (1,6), 
					    (2,4), (4,4),
					    (6,5), (6,6)} {
		\fill[llgray] \point -- +(0, 1) -- +(1, 1) -- +(1, 0);
	}
	\foreach \x in {1, 5} {		
		\fill[llgray] (\x, 3.5) -- (\x, 4) -- (\x+1, 4) -- (\x+1, 3.5) -- cycle;
	}
	\foreach \x in {0, 1, 2, 5, 6} {		
		\fill[llgray] (\x, 7.5) -- (\x, 7) -- (\x+1, 7) -- (\x+1, 7.5) -- cycle;
	}	
	\foreach \y in {5, 6} {		
		\fill[llgray] (0, \y) -- (0, \y+1) -- (-0.5, \y+1) -- (-0.5, \y) -- cycle;
	}
	%
	\foreach \y in {4, 5, 6, 7} {
		\draw[thin] (-0.5, \y) -- (7.5, \y);
	}
	\foreach \x in {0, 1, 2, 3, 4, 5, 6, 7} {
		\draw[thin] (\x, 3.5) -- (\x, 7.5);
	}
	\fill[pattern=north east lines, pattern color=green, opacity=0.4] (1, 5) -- (1, 6) -- (2, 6) -- (2, 5) -- cycle;
	\draw[black, line width=0.3mm, dashed ]  (1.5,5.5) circle (0.3cm);
	\fill[black, line width=0.05mm ]  (1.5,5.5) circle (0.03cm);
    \fill[myRed, line width=0.05mm ]  (4.5,6.3) circle (0.03cm);
    \draw[myRed, line width=0.3mm, dashed ]  (1.5,5.8) -- (4.5,6.3);
    \draw[myRed, line width=0.3mm, dashed ]  (1.548,5.2) -- (4.5,6.3);
    \draw[blue, line width=0.7mm,  ]   plot[smooth, tension=.7] coordinates {(1.2,3.5)(0.2,3.8) (0.25,6.5) (1,7.0) (3,7.1) (5,6.9) (7,6.5) (7,5.5) (6.9,4.5) 
        (5,4.2) (2,3.7)  (1.2,3.5)};
    \node at (1.5, 4.8) {{\color{myGreen2}\Large$q$}};
    \node at (1.5, 5.35) {{\color{black}\normalsize$x_q$}};
    \node at (0.6, 5.5) {{\color{black}\large$B_{\rho}(x_q)$}};
    \node at (4.7, 6.4) {{\color{myRed}\Large$x$}};
    \node at (6.5, 4.7) {{\color{blue}\Large$\hat{D}_q$}};
    \node at (3.5, 5.6) {{\color{myRed}\Large$C_q(x)$}};
\end{tikzpicture} 
\end{center}
\caption{Illustration of a possible setup used in the proof of Lemma \ref{lem_eta_Friedrich}.} 
\label{fig_appendix_label}
\end{figure}
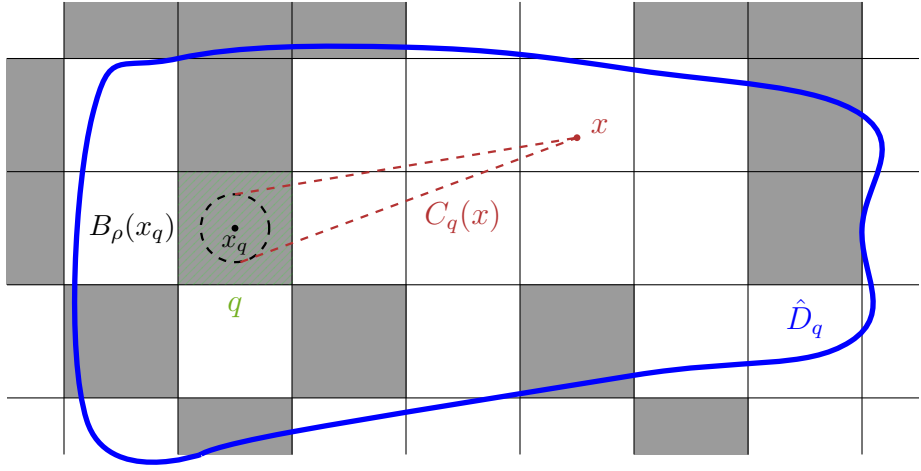

If we select $\hat{D}_q$ as the cube with center $x_{q}$ and edge length $\eps (2L+1)\lesssim \eps L$ 
then we have
\begin{align*}
 1 + \frac{\mbox{dist}(x_{q} , \partial \hat{D}_q) }{\rho} 
 \le 1 + 4 \frac{(2L+1)\eps}{ \eps } \lesssim L
 \qquad
 \mbox{and}
 \qquad
 \int_{\hat{D}_q} |z-x|^{-d+1} \dz \le \eps L.
\end{align*}
This reduces the estimate to
\begin{align}
\label{local-estimates-friedrichs}
\| \eta v \|_{L^2(\hat{D}_q)}
&\lesssim L^{d} \eps L \|\nabla (\eta v) \|_{L^2(\hat{D}_q)}.
\end{align}
Since the union of all such cubes $\hat{D}_q$ with center $x_q$ and edge length $\eps L$ (i.e. $\bigcup_{q \in \calTb} \hat{D}_q$) forms a cover of $D$ with maximal overlap $\kappa_\calT \le L^d$, we conclude from the local estimates \eqref{local-estimates-friedrichs} that it holds
\[
  \| \eta v \|_{L^2(D)}
  \lesssim \kappa_\calT L^{d} \eps L \|\nabla (\eta v) \|_{L^2(D)}. 
  \qedhere
\]
\end{proof}
It is worth to note that the previous estimate can be improved if we make additional structural assumptions on $V$ that allow to decompose $D$ into the disjoint union of certain star-shaped blocks that are centered in a $\alpha$-valley and with a diameter that is of order~$\eps L$. Practically, this would be a moderate assumption, which allows to avoid the overlap parameter $\kappa_\calT$ in the estimate. On the other hand, one would also have to track the size of maximal $\beta$-cubes, which would lead to an estimate of the form
$$
\| \eta v \|_{L^2(D)} \lesssim \Big( 1 + \frac{L}{ L_{\beta}^{\mbox{\tiny min}} } \Big)^{d} \eps\, \big(L+ L_{\beta}^{\mbox{\tiny max}}\big)\, \|\nabla (\eta v) \|_{L^2(D)},
$$
where, analogously to the definition of $L$, 
$$
L_{\beta}^{\mbox{\tiny min}} := \min_{Q\in \calQb} \frac{h_Q}{\eps}
\qquad
\mbox{and}
\qquad
L_{\beta}^{\mbox{\tiny max}} := \max_{Q\in \calQb} \frac{h_Q}{\eps}.
$$
For the sake of simplicity and to avoid technical assumptions, this direction was not pursued in the present paper. 
\end{appendix}
\end{document}